\documentclass[11pt]{amsart}
\usepackage{amsmath,amstext,amssymb,amsopn,amsthm}
\usepackage[top=25mm,bottom=25mm,left=25mm,right=25mm]{geometry}
\usepackage[utf8]{inputenc}

\usepackage[T1]{fontenc}
\usepackage{graphicx}
\usepackage{pdfpages}
\usepackage{xcolor}

\usepackage[square,numbers]{natbib}
\usepackage[bookmarks,colorlinks]{hyperref}

\newcommand{\Rd}{{\mathbb R^d}}
\newcommand{\R}{\mathbb{R}}

\newtheorem{lem}{Lemma}[section]
\newtheorem{thm}[lem]{Theorem}
\newtheorem{dfn}[lem]{Definition}
\newtheorem{cor}[lem]{Corollary}
\newtheorem{prop}[lem]{Proposition}

\newtheorem{rem}{Remark}

\DeclareMathOperator{\diam}{diam}

\DeclareMathOperator{\sgn}{sgn}
\makeatletter
\@addtoreset{equation}{section}
\makeatother

\title{Bregman variation of semimartingales}
\thanks{Accepted for publication in
Electronic Journal of Probability.
DOI: \href{https://doi.org/10.1214/26-EJP1616}
{10.1214/26-EJP1616}.}
\date{}

\author[K.~Bogdan]{Krzysztof Bogdan}
\address{Faculty of Pure and Applied Mathematics, Wroc\l{}aw University of Science and Technology, Wyb. Wyspia\'nskiego 27, 50-370 Wroc\l{}aw, Poland.}
\email{krzysztof.bogdan@pwr.edu.pl}
\author[D.~Kutek]{Dominik Kutek}
\address{Institute of Mathematics, University of Warsaw, ul. Banacha 2, 02-097 Warsaw, Poland.}
\email{d.kutek@mimuw.edu.pl}
\author[K.~Pietruska-Pa\l uba]{Katarzyna Pietruska-Pa\l uba}
\address{Institute of Mathematics, University of Warsaw, ul. Banacha 2, 02-097 Warsaw, Poland.}
\email{kpp@mimuw.edu.pl}
\thanks{Supported by the NCN grant 2018/31/B/ST1/03818.}


\begin{document}
\begin{abstract}
We define Bregman variation of semimartingales. We give its pathwise representation, Itô-type isometry for martingales, and applications to harmonic analysis.
\end{abstract}
\large
\maketitle

\section{Introduction and preliminaries}

The quadratic variation of a c\`adl\`ag semimartingale is a limit of cumulated squared increments (for definitions and details see below). 
It is a key concept in stochastic calculus, with many applications in mathematics, science, and economy. It yields a bridge between stochastic and deterministic integrations, especially for square integrable semimartingales; see, e.g.,  Protter \cite{MR2273672} and Revuz-Yor \cite{MR1725357}. In particular, the quadratic variation is of paramount importance in It\^o calculus, 
stochastic differential equations, finance and statistics. 

In this work, for convex functions $\phi$ satisfying mild growth and regularity conditions and c\`adl\`ag real-valued semimartingales $X_t,\, t\!\in\! [0,\infty)$, we study a similar, but more general concept, the \textit{Bregman variation}, or \textit{$\phi$-variation}, defined by
\begin{equation}\label{eq:var-def}
V^\phi(X)_t:=\phi(X_t)-\int_0^t \phi'(X_{s-})dX_s,\quad t\ge 0.
\end{equation}
As we shall see, 
Bregman variation is tailor-made for studying $\phi$-integrable semimartingales and problems of probability and analysis that involve convex functions more general than the quadratic function. 
Our theory has wider scope than that of the quadratic variation, but also exhibits some similarities. In particular, the Bregman variation of a semimartingale can be approximated by cumulated \textit{Bregman divergences} of the semimartingale. We propose a common framework that unifies observations and results in Probability and Analysis  accumulated in recent years (see \cite{MR4600287}, \cite{MR4589708}, \cite{MR4851904}, \cite{MR4521651}, \cite{MR4622410}, \cite{MR4885983} and \cite{MR3463679}). We also give applications to semigroups and harmonic functions of Markov processes.

Here is the structure of the paper. Below in this section we discuss convex functions on the line, the Bregman divergence, and semimartingales.
In Subsection~\ref{ss.dt}, we briefly address processes in discrete time. In Subsection~\ref{sec:phi-semimart}, we state precise assumptions for the definition of the Bregman divergence \eqref{eq:var-def} for continuous-time (c\`adl\`ag) semimartingales.
In Theorem~\ref{prop:integration}, we give an Itô-type isometry ($\phi$-isometry) for martingales and in Theorem~\ref{t.itm-smooth}, we provide a pathwise formula for $V^\phi(X)$  involving the jumps and the continuous part of the quadratic variation of the semimartingale.
In Section~\ref{s.a}, we present applications and examples. In Appendix~\ref{a.MM}, we give auxiliary estimates for $\phi$-integrable random variables and martingales and in Appendix~\ref{a.si}, we recall relevant notions and results on stochastic integration.

All the considered sets, functions, and measures are assumed to be Borel. The symbol $:=$ indicates definition, e.g., $ \mathbb R_+:=[0,\infty)$, $a\wedge b:=\min\{a,b\}$. Random variables are (finite) real-valued. Unless stated otherwise, stochastic process, in particular continuous-time martingale or semimartingale, is a family of (real-valued) random variables $X(t)=X_t$ indexed by $t\ge 0$, that is, $t\in \R_+$, defined on the same probability space. Equality of random variables is understood almost surely (a.s.). \textit{Constants} in \textit{inequalities} involving functions mean numbers in $(0,\infty)$ that do not depend on the arguments of the functions, see, e.g., \eqref{eq:delta-dwa}.

\subsection{Convex functions}

In this subsection, we briefly discuss convex functions; we refer to Krasnoselskii and Rutickii \cite{MR0126722}, Rao and Ren \cite{MR1113700}, Long \cite{MR1224450}, or Alsmeyr and R\"{o}sler \cite{MR2222745} for more information.
We say that $\phi$  is a \textit{Young function} if $\phi$ is even, continuous, and convex, $\phi:\mathbb R \to \mathbb R_+$, $\lim_{\lambda \to 0^+}\phi(\lambda)/\lambda=0$, and $\lim_{\lambda\to +\infty} \phi(\lambda)/\lambda = +\infty$. In particular, $\phi(0) = 0$.
Let $\phi^*$ be the
Legendre transform of a Young function $\phi$:
\[ \phi^*(\gamma) := \sup\left\{\ \! \gamma\lambda -\phi(\lambda):\ \lambda \in \mathbb R\ \!\right\},\qquad  \gamma\in \R.\]
Then $\phi^*$ is again a Young function and the definition immediately yields \textit{Young's inequality:}
\begin{equation}\label{eq:Young}\tag{YI}
\lambda \gamma \leq \phi(\lambda) + \phi^*(\gamma),\qquad \lambda, \gamma \in \R.
\end{equation}
We say that Young function $\phi$ satisfies the (doubling) condition $\Delta_2$, and we write $\phi\in \Delta_2$,  if there exists a constant $K>0$ such that
\begin{equation}\label{eq:delta-dwa}
\phi(2\lambda)\leq K \phi(\lambda),\quad \lambda >0.
\end{equation}
Below, $K_\phi$ will denote the smallest constant for which \eqref{eq:delta-dwa} holds.

 \begin{dfn}\label{def:moederate}
 A Young function $\phi$ is called \textit{moderate} if $\phi,\phi^* \in \Delta_2$.
 \end{dfn}
Let $\phi'$ be the left derivative of $\phi$; see \cite[(3.1.1)]{MR1224450}. Define
 \begin{equation}\label{e.ddDphi}
 D_\phi := \sup_{\lambda >0} \frac{\lambda \phi'(\lambda)}{\phi(\lambda)},  \qquad d_\phi := \inf_{\lambda >0} \frac{\lambda \phi'(\lambda)}{\phi(\lambda)}.
 \end{equation} 
 We call $D_\phi$ (respectively, $d_\phi$) the Simonenko upper (respectively, lower) index of $\phi$. Due to convexity, $1\le d_\phi \le D_\phi \le +\infty$. 
 By, \cite[Theorem 3.1.1]{MR1224450}, \eqref{eq:delta-dwa} is equivalent to $D_\phi < \infty$
and the condition $\Delta_2$ for $\phi^*$ is tantamount to $d_\phi > 1$; see also \cite[Theorem 4.1 and Theorem 4.3]{MR0126722} and \cite{MR2222745}. In short,
\textit{moderate} 
means
$1 < d_\phi \le D_\phi < \infty$.

\subsection{Luxemburg norm and Orlicz space}

Given a Young function \( \phi \) and a measure \( \mu \) on \( \mathbb R^d \), the \textit{Luxemburg norm} of a function \( f:\mathbb R^d \to \mathbb R \) is
\begin{equation}\label{e.dLn}
\|f\|_\phi:=\|f\|_{L^\phi(\Rd,\mu)} := \inf\left\{K > 0 : \int_{\mathbb R^d}\phi\Big(|f(x)|/K\Big){\rm d}\mu(x) \le 1 \right\}.
\end{equation}
We then define the \textit{Orlicz space}
\[L^\phi:=L^\phi(\Rd,\mu) := \left \{ f:\mathbb R^d \to \mathbb R:\; \|f\|_\phi<\infty\right \}.\]
If \( \mu \) is the Lebesgue measure on \( \Rd \), we  may write \( L^\phi(\Rd):=L^\phi(\Rd,\mu) \).
The Orlicz space
with the Luxemburg norm
is a Banach space, see \cite[Theorem 9.2 and p. 78]{MR0126722}.
For instance, \( L^\phi(\Rd, \mu)=L^p(\Rd,\mu) \) and \( \|\cdot\|_\phi=\|\cdot\|_p \) if \( \phi(\lambda)= |\lambda|^p \) with
\( p \in (1,\infty) \).
Furthermore, 
\begin{equation}\label{eq:orlicz2}
\int_{\Rd} \phi(f/\|f\|_\phi)\,{\rm d}\mu(y)\leq 1;\end{equation}
we even have equality in \eqref{eq:orlicz2}  if \( \phi\in \Delta_2 \).
We note that \( f \in L^\phi(\Rd,\mu) \) does not necessarily imply \( \phi \circ f \in L^1(\Rd,\mu) \), however, it does when \( \phi \in \Delta_2 \).
Indeed, for the latter claim, we consider \( n \in \mathbb N_+ \) that satisfies \( 2^n \ge \|f\|_\phi + 1 \), so by \( \Delta_2 \) and \eqref{e.dLn}, we obtain
\[ \int_{\Rd} \phi(f(y)){\rm d} \mu(y) \le K_\phi^n \int_{\Rd}\phi\left(\frac{f(y)}{\|f\|_\phi + 1}\right){\rm d}\mu(y) \le K_\phi^n < \infty.\]
Thus for $\phi\in \Delta_2$, $f\in L^\phi(\Rd,\mu)$ is equivalent to the \textit{$\phi$-integrability}: $\int_\Rd \phi(f)d\mu<\infty$.
\begin{prop}\label{prop:decomposition}
{\rm (1)} For every Young function \( \phi \), \( L^\phi\subset L^1+L^\infty \).\\
\noindent
{\rm (2)} If \( \phi \) is moderate, with Simonenko indices \( 1<d_\phi\leq D_\phi <
\infty, \) then \( L^\phi \subset L^{d_\phi}+L^{D_\phi} \).
\end{prop}
\begin{proof} Let $\|f\|_\phi>0$.
We note that there is \( \lambda_0\in (0,\infty) \) such that \( \phi(\lambda)\ge \lambda \) if \( \lambda\ge \lambda_0 \). 
Then, \begin{equation}\label{e.lbphi}
\phi(f/\|f\|_\phi)\mathbf{1}_{|f|/\|f\|_\phi\geq \lambda_0} \ge  (f/\|f\|_\phi)\mathbf{1}_{|f|/\|f\|_\phi\geq \lambda_0}. 
\end{equation}
Of course,
$f=f_1+f_\infty$,
where
\( f_1:= f\mathbf{1}_{|f|/\|f\|_\phi\geq \lambda_0} \) and \( f_\infty:=f\mathbf{1}_{|f|/\|f\|_\phi< \lambda_0}\in L^\infty. \)
By \eqref{eq:orlicz2} and \eqref{e.lbphi},
\( f_1\in L^1 \). This proves (1).
In the setting of (2),  \( \phi(\lambda)/\lambda^{d_\phi} \) is increasing and
\( \phi(\lambda)/\lambda^{D_\phi} \) is decreasing on $[0,\infty)$ by \cite[Theorem 3.1.1]{MR1224450}, so
\[ \phi(1)\lambda^{d_\phi}\leq \phi(\lambda)\leq \phi(1) \lambda^{D_\phi},\quad \mbox{for }\lambda\geq 1,\]
\[ \phi(1)\lambda^{D_\phi}\leq \phi(\lambda)\leq \phi(1) \lambda^{d_\phi},\quad \mbox{for }\lambda\leq 1.\]
By \eqref{eq:orlicz2}, 
\( f_{d_\phi}:= f\mathbf{1}_{|f|\geq 1}\in L^{d_\phi} \) and
\( f_{D_\phi}:=f\mathbf{1}_{|f|<1}\in L^{D_\phi}\). Since $f=f_{d_\phi}+ f_{D_\phi}$, we get (2).
\end{proof}

 \subsection{Bregman divergence}
For a finite convex function $\phi: \R\to \R$, we define the Bregman divergence by
 \begin{equation} \label{d.Bd}
 F_\phi(x,y) := \phi(y) - \phi(x) - \phi'(x)(y-x),\quad x,y \in \mathbb R.
 \end{equation}
Here $\phi'$ is the left derivative of $\phi$. By convexity, $F_\phi \ge 0$.
 E.g.,  if  $\phi(x)=|x|^p$, $p>1$,  then
 \[F_\phi(x,y)=F_p(x,y):=|y|^p-|x|^p - p|x|^{p-1} \sgn(x) (y-x),\quad x,y \in \mathbb R,\]
in particular, $F_2(x,y)=(y-x)^2$.

Bregman divergence is commonly used in statistical learning and its
applications, see Amari \cite{MR3495836}, Ay, Jost, Lê, and
Schwachhöfer \cite{MR3701408}, Nielsen and Nock \cite{MR2598413},
or Frigyik, Srivastava and Gupta \cite{MR2589887}. For applications to clustering, see A.~Banerjee, S.~Merugu,
I.~S. Dhillon, and J.~Ghosh
\cite{BanerjeeMeruguDhillonGhosh2005}.
Moreover, it is important for the theory of entropy, see, e.g.,
Wang \cite{MR3206685}. It is also used in PDEs in
Gagliardo\nobreakdash--Nirenberg\nobreakdash--Sobolev inequalities,
see Carrillo, J\"{u}ngel, Markowich, Toscani, and Unterreiter
\cite[p.~71]{MR1853037}, Bonforte, Dolbeault, Nazaret, and Simonov
\cite{MR4896163}, and Borowski and Chlebicka
\cite{MR4518648}.

Our motivation stems from recent applications of Bregman divergences, chiefly $F_p$,  in harmonic analysis and probabilistic potential theory, see  Bogdan, Jakubowski, Lenczewska and Pietruska-Pa\l{}uba \cite{MR4372148}, Bogdan, Grzywny, Pietruska-Pa\l{}uba  and Rutkowski, \cite{MR4589708}, Bogdan, Fafuła and Rutkowski \cite{MR4600287},  Bogdan, Gutowski and Pietruska-Pa\l{}uba \cite{MR4851904},  and in martingale inequalities, see Bogdan and Więcek \cite{MR4521651}.	For instance in \cite{MR4372148}, the corresponding \textit{Sobolev--Bregman forms} yield optimal estimates for certain Feynman--Kac semigroups on $L^p$; for comparison, we refer the reader to Farkas, Jacob and Schilling \cite[(2.4)]{MR1808433} and to Jacob \cite[(4.294)]{MR1873235},  where different integral forms are used to analyze  semigroups in $L^p$. The results of the present paper indicate that Bregman variation of semimartingales, defined in Section~\ref{s.Bv}, is suitable for addressing similar issues in the generality of semimartingales. 

As usual, we say that random variable $Y$ is \textit{$\phi$-integrable} if the $\phi$-moment $\mathbb E \phi(Y)$ is finite, and we say a process is $\phi$-integrable if each of its random variables is $\phi$-integrable.
The following lemma, inspired by \cite[Lemma~2.1]{MR4589708}, indicates that  $F_\phi$ captures the evolution of the $\phi$-moment \textit{additively}.
\begin{lem}\label{lem:infty-integr} Let $\phi$ be a Young function.
Consider a random variable $Y\in L^1(\Omega,\mathcal F, \mathbb P)$ and $\sigma$-algebra $\mathcal A\subset \mathcal F$. If $X=\mathbb E[Y|\mathcal A]$ then
\begin{equation}\label{e.identity}
\mathbb E \phi(Y) = \mathbb E\phi(X) +\mathbb E F_\phi(X,Y).
\end{equation}
In particular, if  $\mathbb E\phi(X)<\infty$ and $\mathbb E\phi( Y)=\infty$, then $\mathbb EF_\phi(X,Y)=\infty$.
\end{lem}
\begingroup
Formally, the difference between the two sides of
\eqref{e.identity} is
\[
\mathbb E[\phi'(X)(Y-X)],
\]
which would vanish upon conditioning on $\mathcal A$. Possible failures
of integrability, however, make this reasoning invalid as it stands:
the expectation above, as well as the conditional expectation used to
justify its vanishing, need not be well defined. The localization
argument below avoids imposing the additional integrability assumptions
that would justify this calculation.

Rather than proving the lemma separately, we establish the following
more general result, which also has a natural interpretation in
prediction theory.

\begin{prop}\label{prop:Bregman-predictor}
Let $\phi:\mathbb R\to\mathbb R$ be a finite convex function. Let
$(\Omega,\mathcal F,\mathbb P)$ be a probability space, let
$\mathcal A\subseteq\mathcal F$ be a sub-$\sigma$-algebra, and let
$Y$ be an integrable random variable. Put
\[
X:=\mathbb E[Y|\mathcal A].
\]
Then, for every real-valued $\mathcal A$-measurable random variable
$Z$,
\begin{equation}\label{e.Bregman-predictor}
\mathbb E F_\phi(Z,Y)
=
\mathbb E F_\phi(Z,X)
+
\mathbb E F_\phi(X,Y),
\end{equation}
where the expectations are understood as extended expectations of
nonnegative random variables. Consequently,
\[
\mathbb E F_\phi(X,Y)
\leq
\mathbb E F_\phi(Z,Y).
\]
If, in addition, $\phi$ is strictly convex and
\[
\mathbb E F_\phi(X,Y)<\infty,
\]
then $X$ is the unique minimizer, up to almost sure equality. The
finiteness condition holds, in particular, if
\[
\mathbb E\phi(Y)<\infty.
\]
\end{prop}

\begin{proof}
From the definition of $F_\phi$, we get the pointwise identity
\begin{align}\label{e.Bregman-three-point}
F_\phi(Z,Y)
&=
F_\phi(Z,X)+F_\phi(X,Y) +
\bigl(\phi'(X)-\phi'(Z)\bigr)(Y-X).
\end{align}
We use a localization argument. For $n\geq1$, let
\[
A_n:=\{|X|\leq n,\ |Z|\leq n\}.
\]
The left derivative of a finite convex function is locally
bounded so the random variable
\[
\bigl(\phi'(X)-\phi'(Z)\bigr)\mathbf 1_{A_n}
\]
is bounded and $\mathcal A$-measurable.
Hence,
\[
\mathbb E\left[
\bigl(\phi'(X)-\phi'(Z)\bigr)(Y-X)\mathbf 1_{A_n}
\right]=0.
\]
Since the mixed term is integrable and has expectation zero, the
pointwise identity \eqref{e.Bregman-three-point} yields
\[
\mathbb E\!\left[F_\phi(Z,Y)\mathbf 1_{A_n}\right]
=
\mathbb E\!\left[F_\phi(Z,X)\mathbf 1_{A_n}\right]
+
\mathbb E\!\left[F_\phi(X,Y)\mathbf 1_{A_n}\right].
\]
Letting \(n\to\infty\) and applying the Monotone Convergence Theorem
gives \eqref{e.Bregman-predictor}. Since \(F_\phi\geq0\), this also
proves the minimizing property of \(X\).


Now suppose that $\phi$ is strictly convex
\cite[p.~253]{MR274683} and that
$\mathbb E F_\phi(X,Y)<\infty$. Then
\(
F_\phi(z,x)>0\)
for all $z\ne x$.
If $\bar X$ is another $\mathcal A$-measurable minimizer, then
\[
\mathbb E F_\phi(\bar X,Y)
=
\mathbb E F_\phi(X,Y)
<\infty.
\]
Hence, by \eqref{e.Bregman-predictor},
\[
\mathbb E F_\phi(\bar X,X)=0.
\]
Therefore $\bar X=X$ almost surely.

Finally, if $\mathbb E\phi(Y)<\infty$, then
\[
\mathbb E F_\phi(0,Y)
=
\mathbb E\left[
\phi(Y)-\phi(0)-\phi'(0)Y
\right]
<\infty.
\]
Taking $Z=0$ in \eqref{e.Bregman-predictor} gives
\[
\mathbb E F_\phi(X,Y)
\leq
\mathbb E F_\phi(0,Y)
<\infty.
\]
\end{proof}

\begin{proof}[Proof of Lemma~\ref{lem:infty-integr}]
Taking $Z=0$ in \eqref{e.Bregman-predictor}, we obtain
\[
\mathbb E F_\phi(0,Y)
=
\mathbb E F_\phi(0,X)+\mathbb E F_\phi(X,Y).
\]
Since
\[
F_\phi(0,w)
=
\phi(w)-\phi(0)-\phi'(0)w
\]
and $\mathbb E X=\mathbb E Y$, the affine terms cancel, and we get
\eqref{e.identity}. The final assertion follows.
\end{proof}

\begin{rem}{\rm
In statistical terminology, $Y$ is the observed outcome,
$\mathcal A$ represents the available information, and an
$\mathcal A$-measurable random variable $Z$ is a predictor. The
quantity $F_\phi(z,y)$ 
is the Bregman divergence used as the loss
incurred when $z$ is used to predict $y$.

The identity \eqref{e.Bregman-predictor}
decomposes the expected loss $\mathbb E F_\phi(Z,Y)$ into the excess loss
\(
\mathbb E F_\phi(Z,X)
\),
resulting from the use of $Z$ in place of $X$, and the minimal expected
loss
\(
\mathbb E F_\phi(X,Y).
\)
Accordingly,
\(
X=\mathbb E[Y|\mathcal A]
\)
is an optimal predictor. 

The closest published antecedent of
Proposition~\ref{prop:Bregman-predictor} is the result of Banerjee,
Guo, and Wang
\cite[Remark~1 and Theorem~1, p.~2665; equation~(1), p.~2666]
{BanerjeeGuoWang2005}. In their notation,
\[
D_\phi(y,z)=F_\phi(z,y),
\]
and their equation~(1) has the same algebraic form as
\eqref{e.Bregman-predictor}. A closely related conditional
Pythagorean identity, under the compatibility condition stated there
as equation~(37), appears in
\cite[Appendix~C, equation~(36) and the display following
equation~(37)]{GottwaldBraun2026}.

Unconditional mean-minimization analogues are given for finitely
supported random vectors in
\cite[Section~3.1, Proposition~1]{BanerjeeMeruguDhillonGhosh2005},
in the language of consistent scoring functions in
\cite[Section~3.1, Theorem~3.1(b)--(c)]{Gneiting2011}, and for random
functions and distributions in
\cite[Theorem~3.1]{MR2589887}. The additional point of
Proposition~\ref{prop:Bregman-predictor} is that
\eqref{e.Bregman-predictor} is formulated as an identity of extended
nonnegative expectations and is justified by localization, without
presupposing integrability of the mixed term in
\eqref{e.Bregman-three-point}.

The same argument extends to vector-valued random variables. More
precisely, let $C\subseteq\mathbb R^n$ be open and convex, let
$\phi:C\to\mathbb R$ be finite and convex, and fix a Borel map
$s:C\to\mathbb R^n$ such that $s(x)\in\partial\phi(x)$, a selection of the subdifferential. Define
\[
F_{\phi,s}(x,y)
:=
\phi(y)-\phi(x)-\langle s(x),y-x\rangle.
\]
If $Y$ is integrable, $X=\mathbb E[Y|\mathcal A]$, and $X$, $Y$,
and $Z$ are $C$-valued, with $Z$ being $\mathcal A$-measurable, then
\eqref{e.Bregman-predictor} remains valid with $F_\phi$ replaced by
$F_{\phi,s}$.
In the localization argument, one additionally bounds
$s(X)$ and $s(Z)$. If $\phi$ is differentiable, one may take
$s=\nabla\phi$, recovering the usual multivariate Bregman divergence.

Finally, if
\[
\mathbb E F_\phi(X,Y)=\infty,
\]
then \eqref{e.Bregman-predictor} implies that
\[
\mathbb E F_\phi(Z,Y)=\infty
\]
for every $\mathcal A$-measurable $Z$. Thus, the minimization statement
remains formally valid, but uniqueness cannot be
concluded even when $\phi$ is strictly convex.}
\end{rem}
\endgroup

\subsection{Semimartingales}
In this subsection, we follow \cite{MR2273672}, but we also like to mention inspirations coming from \cite[p. 334]{MR745449}. We assume that $(\Omega, \mathcal F, \mathbb P, (\mathcal F_t)_{t\geq 0})$  is a filtered probability space and the filtration  fulfills the usual conditions, i.e., it is complete and right continuous \cite[p. 3]{MR2273672}. Unless stated otherwise, all the notions below, especially adaptedness,  refer to this filtration.
Recall that a stochastic process $X:=(X_t, \,t\ge 0)$, is called c\`{a}dl\`{a}g if (a.s.) it has right-continuous paths with left limits and it is called c\`{a}gl\`{a}d if it is left continuous with right limits.
\begin{dfn}
Stochastic process
$X$ is called a semimartingale if $X_t=X_0 + M_t + A_t$, $t \ge 0$, where $M$ is a real-valued adapted c\`{a}dl\`{a}g local martingale, $A$ is a real-valued adapted c\`{a}dl\`{a}g process of finite variation, $M_0=A_0=0$, and $X_0$ is $\mathcal F_0$-measurable.
\end{dfn}
From the definition, every semimartingale $X$ is adapted and c\`{a}dl\`{a}g. The same concerns martingales; in what follows we tacitly or explicitly assume they are adapted and c\`{a}dl\`{a}g. For a semimartingale $X$, we define
$(X_{-})_s:=X_{s-}:= \lim_{r \uparrow s}X_r:=\lim_{r \to s,\, 0\le r<s}X_r$ for $s>0$, \begin{equation}\label{e:zerominus}
X_{0-}:=0,
\end{equation} and
$\Delta X_s:= X_s - X_{s-}$ for $s\ge 0$. Furthermore, if $Y$ is a \textit{c\`{a}gl\`{a}d} adapted process, then the stochastic integral 
\[(Y\cdot X)_t:= \int_{0}^tY_s\,{\rm d}X_s, \quad t\ge 0,\]
is well defined. Here
\[
\int_0^t Y_s\,{\rm d}X_s
=
Y_0X_0+\int_{0+}^t Y_s\,{\rm d}X_s,
\]
where the notation $\int_{0+}^t$ on the right-hand side is that of
\cite[Chapter II]{MR2273672}; see also Appendix~\ref{a.si}. By \cite[Chapter II, Sections 4 and 5]{MR2273672}, the process
$Y\cdot X$ is a semimartingale. If $X$ is a local martingale, then
$Y\cdot X$ is a local martingale as well; see
\cite[Chapter III, Theorem 33]{MR2273672}.

For a stopping time $\tau$ and a semimartingale $X$, we define the \textit{stopped} semimartingale $X^\tau$:
$$X^\tau_t := X_{t\wedge \tau},\quad t \ge 0.$$
By \cite[Chapter II, Theorem 12]{MR2273672},
\begin{equation}\label{e.si}
Y\cdot (X^\tau)=(Y\cdot X)^\tau=(Y\mathbf 1_{[0,\tau]})\cdot X.
\end{equation}

Here is one more notational convention: If $X:=(X^{(1)},\ldots,X^{(d)})$ and $Y:=(Y^{(1)},\ldots,Y^{(d)})$ are continuous-time processes  with values in $\R^d$
and, for each $i$, $X^{(i)}$ is a semimartingale and $Y^{(i)}$ is c\`{a}gl\`{a}d and adapted, then $Y\cdot X := \int Y\,{\rm d}X$ denotes the process
\[\sum_{i=1}^d  Y^{(i)}\cdot X^{(i)}=\sum_{i=1}^d \int Y^{(i)}{\rm d}X^{(i)}.\]
For more details we refer to \cite{MR2273672}; see also Appendix~\ref{a.si}.

\section{Bregman variation of stochastic processes}\label{s.Bv}

\subsection{Discrete time}\label{ss.dt}
As a preparation, we discuss the discrete-time  Bregman variation.
\begin{dfn}
Let $\phi$ be a (finite) convex function on $\R$ and let $\phi'$ be the left derivative of $\phi$.
The $\phi$-variation of a (real-valued) process $X_0, X_1,\ldots$ is
\begin{align}\label{e.VbyI}
V^\phi(X)_n  & := 
 \phi(X_n) 
- \sum_{0< k\le n}\phi'(X_{k-1})(X_{k}-X_{k-1}),\quad n=0, 1,\ldots.
\end{align}
\end{dfn}
Of course,  each $V^\phi(X)_n$ is adapted to $\mathcal F_n^X := \sigma(X_0,\ldots,X_n)$ and
\begin{align}
V^\phi(X)_n
&= \phi(X_0)+\sum_{0< k\le n}(\phi(X_{k})-\phi(X_{k-1})) - \sum_{0< k\le n}\phi'(X_{k-1})(X_{k}-X_{k-1})\nonumber\\
&=\phi(X_0)+\sum_{0< k\le n} F_\phi(X_{k-1},X_{k}).\label{e.VbyF}
\end{align}
Since $F_\phi \ge 0$, the process $V^\phi(X)$ is nondecreasing. If $\phi\ge 0$, in particular if $\phi$ is a Young function, then $V^\phi(X)\ge 0$, too.
By Lemma~\ref{lem:infty-integr} and \eqref{e.VbyF}, we get the following result.
\begin{cor}\label{cor:equal}
For a Young function $\phi$ and a martingale $X_0, \ldots, X_n$, $\mathbb E\phi(X_n)=\mathbb E V^\phi(X)_n$.
\end{cor}
Clearly, the sum $\sum_{0\le k<n}\phi'(X_k)(X_{k+1}-X_k)$ in \eqref{e.VbyI} is a discrete analogue of the stochastic integral $\int_0^t \phi'(X_{s-}){\rm d}X_s$ in \eqref{eq:var-def}, which prepares the scene for the next subsection.

\subsection{Continuous time}
\label{sec:phi-semimart}

We come back to the main setting of the paper. Suppose that $X$ is a continuous-time semimartingale. As usual, $C^1$ means \textit{continuously differentiable}.
Recall that every differentiable convex function on $\mathbb R$ is
continuously differentiable; see \cite[Corollary~25.5.1]{MR274683}.
Thus, for a Young function, differentiability is equivalent to being
of class $C^1$.
\begin{dfn}\label{def-variation-cont}
Let $\phi$ be a $C^1$ Young function
and $X$ a continuous-time semimartingale. The Bregman variation, or $\phi$-variation, of $X$ is the stochastic process defined by \eqref{eq:var-def}.
\end{dfn}
In short, $V^\phi(X) := \phi(X) -\phi'(X_{-})\cdot X$. Since $\phi'(X_{-})$ is c\`{a}gl\`{a}d, the stochastic integral
in \eqref{eq:var-def} is a semimartingale  \cite[Theorem II.19]{MR2273672}.
\begin{rem}\label{rem:integralnotation}{\rm 
For clarity, we point out that, according to the notation of
\cite[Chapter II, Section 5]{MR2273672} and the convention \eqref{e:zerominus}, we have 
\[\int_{0}^t\phi'(X_{s-})\,{\rm d}X_s:= \int_{[0,t]}\phi'(X_{s-})\,{\rm d}X_s=\int_{(0,t]}\phi'(X_{s-})\,{\rm d}X_s\] because $\phi'(X_{0-})=\phi'(0)=0$.\qed
}
\end{rem}
Further, by \eqref{e.si}, 
\begin{equation}\label{e.gwiazdka}
V^\phi (X^\tau)=\left(V^\phi(X)\right)^\tau
\end{equation}
for each stopping time $\tau$.
As expected, $V^\phi(X)$
can be approximated by discrete sums: For \textit{random partitions} $\pi^{(n)}=  \{0=\tau_0^{(n)}\le \tau_1^{(n)} \le \ldots \le \tau_{k_n}^{(n)}\}$ \textit{tending to the identity}, we have
\[\int_0^t\phi'(X_{s-}) {\rm d}X_s = \lim_{n\to\infty}\sum_{k=0}^{k_n-1}\phi'(X_{t \wedge \tau_{k}^{(n)}})\left(X_{t \wedge \tau_{k+1}^{(n)}}-
X_{t \wedge \tau_{k}^{(n)}}\right)\] in the \textit{ucp} sense; see \cite[Chapter II, Section 5]{MR2273672} or Theorem \ref{th:approx} below.
Also in \textit{ucp},
\begin{equation*}\phi(X_t)-\phi(X_0) =\lim_{n \to \infty} \sum_{k=0}^{k_n-1} \left(\phi(X_{t \wedge \tau_{k+1}^{(n)}})-
\phi(X_{t \wedge \tau_{k}^{(n)}})\right).
\end{equation*}
Hence, by \eqref{d.Bd} and \eqref{eq:var-def},
\begin{align}\label{e.da}
V^\phi(X)_t & = \phi(X_0) +
\lim_{n\to\infty} \sum_{k=0}^{k_n-1} F_\phi(X_{t \wedge \tau_{k}^{(n)}},X_{t \wedge \tau_{k+1}^{(n)}}),
\end{align}
that is, the continuous-time $\phi$-variation is a \textit{ucp} limit of discrete $\phi$-variations for $X_{t \wedge \tau_{0}^{(n)}}, X_{t \wedge \tau_{1}^{(n)}},\ldots$; see Subsection~\ref{ss.dt}.
\begin{cor}\label{c.pBv}
The Bregman variation is c\`{a}dl\`{a}g, nondecreasing, nonnegative, and adapted.
\end{cor}
Bregman variation enjoys an It\^o-type isometry, or $\phi$-isometry. 
\begin{thm}\label{prop:integration}
Suppose that $(X_t, t \ge 0)$ is a
c\`{a}dl\`{a}g local martingale and $\phi$ is a $C^1$ moderate Young function. The following two conditions are equivalent:

{\rm (1)} $(X_t)_{t\geq 0}$ is a $\phi$-integrable martingale.

{\rm (2)} For every $t\ge 0$,  $\mathbb E V^\phi(X)_t<\infty.$\\
If they hold, then
\begin{equation}\label{eq:ito-isometry}
\mathbb E \phi(X_t) = \mathbb E V^\phi(X)_t,\quad t\ge 0.
\end{equation}

\end{thm}
\begin{proof}

Assume that (1) holds. The process $\phi(X)-V^\phi(X)={\phi'(X_{-})\cdot X}$
is a local martingale. Let $T_N\nearrow \infty$ be its localizing sequence of stopping times. Thus, for $N=1,2,...,$ the process
\begin{equation}\label{e.m0e}
\phi(X_{t\wedge T_N})-V^\phi(X)_{t\wedge T_N},\quad t\ge 0,
\end{equation}
is a martingale with expectation $0$. In particular,
\begin{equation}\label{e.sM}
\mathbb E\phi(X_{t\wedge T_N}) = \mathbb E V^\phi(X)_{t \wedge T_N},\quad t\ge 0.
\end{equation}
Fix $t\in [0,\infty)$ and let $N\to \infty$. We have $t\wedge T_N \nearrow t$, in fact, $t\wedge T_N = t$ for $N$ large enough (depending on $\omega$). By Corollary~\ref{c.pBv} and
the Monotone Convergence Theorem,
\begin{equation}\label{eq:iso-t}
\mathbb E V^\phi(X)_t = \lim_{n\to\infty} \mathbb  E V^\phi(X)_{t\wedge T_n} = \lim_{n\to\infty} \mathbb E \phi(X_{t\wedge T_n}) =  \mathbb E \lim_{n\to\infty}\phi(X_{t\wedge T_n}) = \mathbb E \phi(X_t). 
\end{equation}
The penultimate equality is legitimate since  $\mathbb E\phi(X_t)^*\leq \left(\frac{d_\phi}{d_\phi-1}\right)^{D_\phi} \mathbb E\phi(X_t)<\infty$ by Corollary \ref{cor:doob-cont}, so the Dominated Convergence Theorem can be applied\footnote{As usual, $X_t^* = \sup_{0 \le s \le t}X_s$ and $\phi(X_t)^*=\sup_{0\le s\le t} \phi(X_s) = \phi(|X_t|^*)$.}. Thus, (1) implies (2). Moreover, (1) implies the equation \eqref{eq:ito-isometry}.
\smallskip

Assume now that (2) holds. Since $(X_t)_{t \ge 0}$ is a local martingale, we can again find a localizing sequence of stopping times $T_N\nearrow \infty$ such that both the process \eqref{e.m0e} and $X^{T_N}$ are martingales for each $N=1,2,...$. 
Then, $\mathbb E\phi(X_{t \wedge T_N}) = \mathbb EV^\phi(X)_{t \wedge T_N}$ for all $t \ge 0$ and $N =1,2,\ldots$. By Fatou's lemma and Corollary~\ref{c.pBv}, 
\[ \mathbb E\phi(X_t) \le \liminf_{N \to \infty} \mathbb E\phi(X_{t \wedge T_N}) = \liminf_{N \to \infty} \mathbb EV^\phi(X_{t \wedge T_N}) \le \mathbb EV^\phi(X)_t < \infty,\] 
so $X$ is $\phi$-integrable. It remains to check that $X$ is a martingale. By Corollary \ref{cor:doob-cont} and the implication (1) $\Rightarrow$ (2) applied to the martingales $X^{T_N}$,
\begin{eqnarray*}
    &&  \mathbb E \sup_{s \le t}\phi(X_s) = \mathbb E\sup_N\sup_{s\le t\wedge T_N}\phi(X_s)= \sup_N \mathbb E  \sup_{s \le t\wedge T_N} \phi(X_s^{T_N}) = \sup_N \mathbb E \sup_{s\leq t}  
\phi(X_s^{T_N})\\
& \le &  C_\phi \sup_{N} \mathbb E\phi(X_t^{T_N}) = C_\phi \sup_{N} \mathbb EV^\phi(X)_{t \wedge T_N} = C_\phi \mathbb EV^\phi(X)_t < \infty.
\end{eqnarray*}
In particular, $\sup_{s \le t}|X_s| \in L^1(\mathbb P)$, see Lemma \ref{le:embedding}. Then, by the Dominated Convergence Theorem, we conclude that $X_r = \mathbb E[X_s | \mathcal F_r]$, so $X$ is a true martingale.
\end{proof}

\begingroup
\begin{rem}{\rm
For a general semimartingale, the approximation \eqref{e.da} shows
that\linebreak $V^\phi(X)_t-\phi(X_0)$ measures the cumulative change of $X$ up
to time $t$, with the size of each change measured by $F_\phi$. For
$\phi(x)=x^2$, this reduces to the usual quadratic variation.

For martingales, the discussion following
Lemma~\ref{lem:infty-integr} and
Proposition~\ref{prop:Bregman-predictor} gives an additional
prediction-theoretic interpretation. Suppose that $X$ is a
$\phi$-integrable martingale, fix $0\leq s\leq t\leq T$, and put
\[
L_r(T):=\mathbb E F_\phi(X_r,X_T),
\qquad 0\leq r\leq T.
\]
Since
\[
X_r=\mathbb E[X_T\mid\mathcal F_r],
\]
the quantity $L_r(T)$ is the expected Bregman loss when $X_r$ is used
to predict $X_T$. Proposition~\ref{prop:Bregman-predictor} and the
$\phi$-isometry give for $0\le s\le t\le T$,
\[
L_s(T)-L_t(T)
=
\mathbb E F_\phi(X_s,X_t)
=
\mathbb E\left[V^\phi(X)_t-V^\phi(X)_s\right].
\]
Thus, the expected increase of the Bregman variation over $[s,t]$
equals the amount by which the expected loss decreases when the
predictor of $X_T$ is updated from $X_s$ to $X_t$.
}
\end{rem}
\endgroup

Note that \eqref{eq:ito-isometry} needs not hold when deterministic times $t$ are replaced by stopping times $\tau$. This can be seen in the case of the quadratic variation of the Brownian motion $(W_t, t \ge 0)$ and the stopping time $\tau = \inf\{t \ge 0 : W_t = 1\}$. However, \eqref{eq:ito-isometry} clearly holds for $T \wedge \tau$ with arbitrary $T \ge 0$. So, if we make an additional integrability assumption $V^\phi(X)_\tau \in L^1(\mathbb P)$, we can pass to the limit $t\to\infty$ and obtain the following result.

\begin{cor}\label{c.irt} If $\phi$ is a $C^1$ moderate Young function,  $(X_t, t\ge 0)$ is a c\`adl\`ag local martingale, $\tau$ is an a.s. finite stopping time, and $\mathbb EV^\phi(X)_\tau < \infty,$  then $\mathbb E\phi(X_\tau) = \mathbb EV^\phi(X)_\tau.$

\end{cor}

\begin{proof}
By \eqref{e.gwiazdka}, for every $T\geq 0$,
\[
\mathbb E V^\phi(X^\tau)_T
=
\mathbb E V^\phi(X)_{T\wedge\tau}
\leq
\mathbb E V^\phi(X)_\tau
<\infty.
\]
Hence, Theorem~\ref{prop:integration} shows that $X^\tau$ is a
$\phi$-integrable martingale and
\[
\mathbb E\phi(X_{T\wedge\tau})
=
\mathbb E V^\phi(X)_{T\wedge\tau},
\qquad T\geq0.
\]
So, by Corollary \ref{cor:doob-cont}, \[ \mathbb E\sup_{t \ge 0}\phi(X_t^\tau) = \sup_{T \ge 0}\mathbb E\sup_{t \le T}\phi(X_t^\tau) \le C_\phi \sup_{T \ge 0}\mathbb E\phi(X_T^\tau) = C_\phi \sup_{T \ge 0}\mathbb EV^\phi(X)_{T \wedge \tau} \le C_\phi \mathbb EV^\phi(X)_\tau < \infty.\] By the Dominated Convergence Theorem and Monotone Convergence Theorem, \[ \mathbb E\phi(X_\tau) = \lim_{t \to \infty} \mathbb E\phi(X_{t \wedge \tau}) = \lim_{t \to \infty} \mathbb EV^\phi(X)_{t \wedge \tau} = \mathbb EV^\phi(X)_\tau. \qedhere\]
\end{proof}

For (closed) martingales with finite time horizon, we proceed as follows.
If  $T>0$ and $(Y_t, t \in [0,T])$ is a c\`{a}dl\`{a}g martingale with respect to a complete, right-continuous filtration $\mathcal G^Y_t$, $t\in [0,T]$, we define $X_t := Y_t$, $\mathcal F_t:=\mathcal G_t$ for $t \le T$ and $X_t := Y_T$, $\mathcal F_t:=\mathcal G_T$ for $t>T$. Then, $V^\phi(Y)_t=V^\phi(X)_t$ for $t\in [0,T]$
and $V^\phi(Y)^\tau=V^\phi(Y^\tau)$ for every stopping time $\tau$. 

\begin{cor}
    \label{l.mT} If $\phi$ is a $C^1$ moderate Young function and $(Y_t, t \in [0,T])$ is a c\`{a}dl\`{a}g martingale with $\mathbb EV^\phi(Y)_T < \infty$,
then $\mathbb E\phi(Y_T) = \mathbb EV^\phi(Y)_T $.
\end{cor}
\begin{proof}
The martingale $X$ constructed above is c\`{a}dl\`{a}g. The result follows from Theorem~\ref{prop:integration}, since $\phi(Y_T)=\phi(X_T)$ and $V^\phi(X)_T = V^\phi(Y)_T$.
\end{proof}

\subsection{Connection to quadratic variation}\label{s.oqv}
For $\phi(x)=x^2$,
we get the usual \textit{quadratic variation} of the semimartingale $X$,
\[[X]_t := [X,X]_t:=X_t^2-\int_0^t 2X_{s-}{\rm d}X_s,\quad t\ge 0,\]
see \cite[II.6]{MR2273672}. Following \cite[II.6]{MR2273672}, we denote by $[X]_t^c$ the continuous part of the increasing function $[X]_t$ which starts from zero, that is, $[X]^c_0 = 0$.
By \cite[Chapter II, p. 70]{MR2273672}, the quadratic variation $[X]_t$ can only jump with $X$:
 \[[X]_t = [X]_t^c + X_0^2 + \sum_{0 < s \le t,\, \Delta X_s\neq 0}(\Delta X_s)^2,\qquad t\ge 0.\]
As usual, for two semimartingales $X$ and $Y$, their quadratic co-variation is
\[[X,Y] :=\frac{1}{4}\left([X+Y] - [X-Y]\right).\]
If $X$ is a pure-jump semimartingale, that is, $[X]^c = 0$, and if  $Y$ is an arbitrary semimartingale,
then by \cite[Chapter II,
Theorem 28]{MR2273672},
\begin{align}\label{eq:covar-jump}
 [X,Y]_t= X_0Y_0 + \sum_{0 < s\leq t,\,  \Delta X_s\neq 0} \Delta X_s \Delta Y_s.
\end{align}

 Let $X$ be a semimartingale and $\phi$ be a $C^1$ moderate Young function.
Let $L^y(X)=\left(L^y_t(X),\;t \ge 0\right)$ be the
local time of $X$ at level $y\in\mathbb R$; see \cite[IV.7]{MR2273672}.
By $\mu_\phi$ we denote the second distributional derivative of $\phi$, which is a measure on $\mathbb R$.
We recall a version of the It\^o--Tanaka--Meyer formula from \cite[Chapter IV, Theorem 71]{MR2273672}:
  \begin{eqnarray}\label{eq:ito-tanaka-meyer}
  \phi(X_t)   &=& \phi(X_0) + \int_0^t \phi'(X_{s-})dX_s + \frac{1}{2}\int_{\mathbb R}L_t^y(X)d\mu_\phi(y)\\
  \nonumber
  &&+ \sum_{0 <s \le t,\, \Delta X_s\neq 0}F_\phi(X_{s-},X_s), \quad t \ge 0.
  \end{eqnarray}
By \eqref{eq:var-def}, we get the following result.
\begin{cor}\label{c.fBv} If $X$ is a semimartingale and $\phi$ is a $C^1$ moderate Young function, then
\begin{equation}
\label{e.fV}
V^\phi(X)_t = \phi(X_0) + \frac{1}{2} \int_{\mathbb R}  L^y_t(X)\,{\rm d}\mu_\phi(y) +\sum_{0< s\leq t,\, \Delta X_s\neq 0} F_\phi(X_{s-},X_s), \quad t \ge 0.
\end{equation}
\end{cor}
In Theorem~\ref{t.itm-smooth} below, for $\phi$ regular enough,
we reformulate \eqref{e.fV} using the quadratic variation instead of local times and the following result (\cite[Corollary 1, p. 219]{MR2273672}).
\begin{lem}\label{l.ltqv}
Let $X$ be a semimartingale and let $g$ be either nonnegative or bounded.
 Then,
 \[ \int_{\mathbb R}g(y)L_t^y(X){\rm d}y = \int_{0}^t g(X_{s-}){\rm d}[X]^c_s, \quad t\ge 0. \]
\end{lem}
\begin{proof} The proof for bounded  functions $g$ can be found in \cite[Corollary 1, p. 219]{MR2273672}.
The extension to nonnegative $g$ is obtained by the Monotone Convergence Theorem.
\end{proof}

\begin{thm}\label{t.itm-smooth}
If  $\phi$ is a Young function with absolutely continuous derivative, $X$ is a semimartingale,
 and $\tau<\infty$ is a stopping time, then 
\begin{eqnarray}\label{eq:itm-smooth}
   V_\tau^\phi(X)  &=& \phi(X_0)+ \frac{1}{2}\int_0^\tau \phi''(X_{s-}){\rm d}[X]^c_s + \sum_{0<s\le \tau,\, \Delta X_s\neq 0} F_\phi(X_{s-},X_s) \\
\nonumber
&& = \frac{1}{2}\int_0^\tau \phi''(X_{s-}){\rm d}[X]^c_s + \sum_{0\le s\le \tau,\, \Delta X_s\neq 0} F_\phi(X_{s-},X_s).
  \end{eqnarray}
\end{thm}
\begin{proof}
We have $\phi'(b)-\phi'(a)=\int_a^b \phi''(t){\rm d}t$, $a,b \in \mathbb R_+$, where $\phi''\ge 0$ almost everywhere, so
\[
\int_{\mathbb R}L_t^y(X){\rm d}\mu_\phi(y)= \int_{\mathbb R}L_t^y(X)\phi''(y){\rm d}y = \int_0^t \phi''(X_{s-}){\rm d}[X]_s^c,\quad t \in \mathbb R_+,
\]
by Lemma~\ref{l.ltqv}.
Using \eqref{e.fV}, we get the first equality. Since $\phi(0)=0$, $F(0,X_0)=\phi(X_0)$, which gives the second equality.
\end{proof}
\begin{cor}\label{c.t-} In the setting of Theorem~\ref{t.itm-smooth}, $V^\phi(X)_{t}=V^\phi(X)_{t-}+F_\phi(X_{t-},X_{t}),\quad t>0.$
\end{cor}
\begin{proof}Clearly,
   \begin{align}\label{eq:itm-smooth-2}
   V_t^\phi(X)
&= \phi(X_0)+\frac{1}{2}\int_0^t \phi''(X_{s-}){\rm d}[X]^c_s + \sum_{0<s<t,\, \Delta X_s\neq 0} F_\phi(X_{s-},X_s)+F_\phi(X_{t-},X_t),
\end{align}
which yields the statement, since ${\rm d}[X]^c$ has no atoms.
\end{proof}

We conclude this section with a result on the $\phi$-variation of a sum of independent martingales, which generalizes the case of quadratic variation. We recall that, given $\phi \in \Delta_2$, $K_\phi$ is (optimal) such that $\phi(2x) \le K_\phi \phi(x)$ for every $x \ge 0$.

\begin{prop}\label{prop:var-of-sum} Assume that $\phi$ is a $C^1$ moderate Young function and $X, Y$ are two independent, c\`adl\`ag $\phi$-integrable martingales with $X_0=Y_0=0$. Then, \[ \frac{1}{2} \mathbb E\left[ V^\phi_t(X) + V^\phi_t(Y) \right] \le \mathbb EV^\phi_t(X+Y) \le \frac{K_\phi}{2} \mathbb E\left[ V^\phi_t(X) + V^\phi_t(Y) \right], \quad t \ge 0.\] If for some $s>0$, $\mathbb EV_s^\phi(X+Y) = \mathbb E[V_s^\phi(X) + V_s^\phi(Y)]$ for all pairs $X, Y$ of independent, bounded c\`adl\`ag martingales with $X_0=Y_0=0$, 
then $\phi(x)=c x^2$ for some $c>0$.

\end{prop}

\begin{proof}

By the $\phi-$isometry, $\mathbb EV^\phi_t(X) = \mathbb E\phi(X_t)$, $\mathbb EV^\phi_t(Y) = \mathbb E\phi(Y_t)$ and $\mathbb EV^\phi_t(X+Y) = \mathbb E\phi(X+Y)$. Hence, it suffices to prove \[ \frac{1}{2} \mathbb E\left[ \phi(X_t) + \phi(Y_t) \right] \le \mathbb E\phi(X_t+Y_t) \le \frac{K_\phi}{2} \mathbb E\left[ \phi(X_t) +\phi(Y_t) \right], \quad t \ge 0.\]
Since $\phi$ is convex and $X,Y$ are independent and centered, we obtain \[ \mathbb E\phi(X_t+Y_t) = \int \mathbb E\phi(X_t + y) {\rm d}\mu_{Y_t}(y) \ge \int \phi( \mathbb E[X_t] +y) {\rm d}\mu_{Y_t}(y) = \int \phi(y){\rm d}\mu_{Y_t}(y) = \mathbb E\phi(Y_t).\] We proceed analogously with $\mathbb E\phi(X_t)$ and so the lower bound follows. On the other hand, \[ \phi(X_t+Y_t) = \phi\left( \frac{2X_t}{2} + \frac{2Y_t}{2}\right) \le \frac{1}{2}\left( \phi(2X_t) + \phi(2Y_t) \right) \le \frac{K_\phi}{2}\left( \phi(X_t) + \phi(Y_t)\right).\] So, taking expectation of both sides, we get the upper bound.

\smallskip 

Now, assume that $\mathbb EV_s^\phi(X+Y) = \mathbb E[V_s^\phi(X) + V_s^\phi(Y)]$ holds for some $s>0$ and all pairs $X,Y$ of independent, bounded c\`adl\`ag martingales starting from $0$. Then $\mathbb E\phi(X_s+Y_s) = \mathbb E\phi(X_s) + \mathbb E\phi(Y_s)$, too. Let $x,y \in \mathbb R$ and let $\varepsilon_1,\varepsilon_2$ be independent Rademacher random variables. Define c\`adl\`ag martingales $X$,$Y$ by the formula $X_t = 0 = Y_t$ for $t<s$ and $X_t = x \varepsilon_1$, $Y_t = y \varepsilon_2$ for $t \ge s$. Then, by the symmetry of $\phi$, \[ \mathbb E\phi(X_s+Y_s) = \mathbb E\phi(x\varepsilon_1 + y\varepsilon_2) = \frac{1}{2}\left(\phi(x+y) + \phi(x-y)\right). \] Of course, \[ \mathbb E[\phi(X_s) + \phi(Y_s)] = \phi(x) + \phi(y).\] 
Hence, $\phi$ satisfies the functional equation
\[
\phi(x+y)+\phi(x-y)=2\phi(x)+2\phi(y).
\]
Its continuous solutions are $\phi(x)=c x^2$, $c\in\mathbb R$.
Since $\phi$ is a Young function, necessarily $c>0$.
\end{proof}

In the next section we present applications of \eqref{e.fV}, \eqref{eq:itm-smooth},
and the $\phi$-isometry to recent problems of harmonic analysis.

\section{Applications}\label{s.a}
 In this section we present applications of the $\phi$-isometry to analysis of semigroups and harmonic functions of some Markov processes. 
Below, we assume that $\phi$ is a $C^1$ moderate Young function.
In Theorem~\ref{th:Hardy-general}, Proposition~\ref{p.piD}, and the
corollary following that proposition, we additionally assume that
$\phi'$ is absolutely continuous.
\subsection{Parabolic  Hardy--Stein identity}\label{s.HS}

 We shall use the $\phi$-variation to prove the so-called parabolic Hardy--Stein identity in $L^\phi(\Rd)$ for semigroups of L\'{e}vy processes satisfying certain mild assumptions; see Theorem~\ref{th:Hardy-general}. The identity has been previously proved when $\phi(t)=t^p$ with $p\in (1,\infty)$ for some pure-jump symmetric (see Ba\~{n}uelos, Bogdan and Luks \cite{MR3556449}) and nonsymmetric (see Ba\~{n}uelos and Kim \cite{MR3994925}) L\'{e}vy processes as well as  for the Brownian motion \cite{MR4851904}.
So the novelty of Theorem~\ref{th:Hardy-general} is that it can handle L\'{e}vy processes that have nontrivial \textit{both} diffusion \textit{and} jump parts and we allow rather general functions $\phi$. We expect that our method should extend to other Markovian semigroups; see Gutowski and Kwa\'snicki \cite{MR4885983} and Gutowski \cite{MR4622410} for results in the case of $\phi(t)=t^p$.
\subsubsection{Assumptions and preliminaries}\label{sec:Levy}
We consider a (L\'evy) measure $\nu$ on  $\Rd$
satisfying
	\[ \int_{\Rd}\left(|z|^2\wedge 1\right)\nu({\rm d}z) <\infty.\]
Assume that $\nu$ is symmetric: $\nu(A)=\nu(-A)$, $A\subset \Rd$. Let
$Q$ be a real symmetric nonnegative-definite $d\times d$ matrix. Its symmetric, nonnegative-definite square root will be denoted $\mathcal R=(r_{ij})_{i,j=1}^d:=\sqrt Q$. Thus, $Q = \mathcal R^2$. Let
	\begin{equation*}
		\psi(\xi): = \frac{1}{2} \langle Q\xi, \xi\rangle + \int_{\Rd} \left(1-\cos\langle\xi, x\rangle\right)\nu({\rm d}x),\quad \xi\in \Rd,
	\end{equation*}
where $\langle\xi , x\rangle$ is the Euclidean scalar product in $\Rd$.
There exists an $\Rd$-valued L\'evy process $(Z_t,\, t\ge 0)$ satisfying the following L\'evy-Khinchine formula:
\[\mathbb E_x{\rm e}^{i \langle \xi, Z_t \rangle}={\rm e}^{i \langle \xi, x \rangle -t\psi(\xi)},\quad \xi\in\Rd, \,t\ge 0,\, x\in \Rd.\]
Here and below, we adopt the Markovian notation, with probabilities $\mathbb P_x$ for the process 
\textit{starting from} $x\in\Rd$, 
the corresponding expectations $\mathbb E_x$, and the \textit{usual filtration} $\mathcal F^Z_t$, $t\ge 0$;
see \cite[Theorem I.31]{MR2273672} or
Sato \cite{MR1739520}. Recall that $Z$ is a Hunt process, that is, $Z$ is c\`adl\`ag, quasi-left continuous, and strong Markov with right-continuous complete filtration, see \cite[IV.7]{MR850715} or \cite[Chapter 3]{MR648601}.
To simplify technical considerations, apart from the symmetry of $\nu$, in what follows, we also assume the Hartman-Wintner condition:
\begin{equation}\label{eq:HaWi}\tag{HW}
\lim_{|\xi|\to \infty}
\frac{\psi(\xi)}{\log|\xi|}=\infty.
\end{equation}
 We let
	\begin{equation}
		\label{eq:Fi}
		p_t(x):=(2\pi)^{-d}\int_{\Rd} {\rm e}^{-i \langle \xi, x\rangle}{\rm e}^{-t\psi(\xi)}\,{\rm d}\xi,\quad t>0, \ x\in \Rd.
	\end{equation}
	We see that $p_t$ is symmetric; $p_t(-x)=p_t(x)$. It is a convolution semigroup of probability densities, i.e., $p_t*p_s=p_{t+s}$. By \eqref{eq:Fi} and \eqref{eq:HaWi}, $\|p_t\|_\infty = p_t(0) \to 0$ as $t\to \infty$. Furthermore, $p_t(x)$ is smooth (i.e., $C^\infty$) in $x$, with all the derivatives bounded and integrable; see \cite[Theorem 2.1]{MR3010850}.
We let $p_t(x,y):=p_t(y-x)$. Of course, $p_t(x,y)=p_t(y,x)$
and
\begin{equation}\label{e.sk}
\int_{\mathbb R^d}p_t(x,y){\rm d}x = \int_{\mathbb R^d}p_t(x,y){\rm d}y = \int_{\mathbb R^d}p_t(x){\rm d}x = 1.
\end{equation}
The function $(t,x,y) \mapsto p_t(x,y)$ is a transition density of $Z,$ the Hunt process associated with $\Psi$. In particular,
the semigroup of $Z$ is given by
\begin{equation}\label{eq:semigroup-def}
P_t f(x) := \mathbb E_xf(Z_t)= \int_{\mathbb R^d}f(y)p_t(x,y){\rm d}y = f*p_t(x), \quad t>0, \  x \in \mathbb R^d,
\end{equation}
for nonnegative or integrable function $f$. Moreover, we let $P_0f:=f$.
If $f \in L^\phi(\mathbb R^d)$, then $P_tf\in L^\phi(\mathbb R^d)$, in fact, the operators $P_t$ are contractions on $L^\phi$. Indeed, for $t>0$, $x \in \Rd$,
\begin{equation}\label{e.exfzt}
\mathbb E_x\phi(f(Z_t)) = \int_{\Rd} \phi(f(y))p_t(x,y){\rm d}y \le p_t(0) \int_{\Rd}\phi(f(y)){\rm d}y < \infty,
\end{equation}
hence, by Jensen's inequality and \eqref{eq:semigroup-def},
 \begin{align}
\int_{\mathbb R^d}\phi(P_tf(x)){\rm d}x &= \int_{\mathbb R^d}\phi(\mathbb E_xf(Z_t)){\rm d}x \leq \int_{\mathbb R^d}\mathbb E_x\phi(f(Z_t)){\rm d}x 
 = \int_{\mathbb R^d} \phi(f(y)){\rm d}y.\label{e.Y}
 \end{align}
 The last equality in \eqref{e.Y} follows from Tonelli's theorem and \eqref{e.sk}, which indeed give
\begin{equation}\label{eq:later}
\int_{\mathbb R^d}\int_{\mathbb R^d} \phi(f(y))p_t(x,y){\rm d}x{\rm d}y=\int_{\mathbb R^d} \phi(f(y)){\rm d}y.
\end{equation}
Of course, \eqref{e.Y} remains true for $t=0$.
Considering $f/\|f\|_\phi$ with $\|f\|_\phi>0$, by \eqref{eq:orlicz2} we get
\begin{equation}\label{eq:contraction} \|P_t f\|_\phi\leq \|f\|_\phi, \quad t \ge 0. \end{equation}

We say that function $g:I \times U \to \mathbb R$ is in class $C^{1,2}(I \times U) = C^{1,2}$ (with $I \subset \mathbb R, U \subset \mathbb R^d$ open), if for every $t \in I$ the function $U \ni x \mapsto g(t,x)$ is $C^2$ and for every $x \in U$ the function $I \ni t \mapsto g(t,x)$ is $C^1$.

\begin{lem}\label{l.rPtf}
The map $(0,\infty) \times \Rd\ni (t,x) \mapsto P_tf(x)$ is $C^{1,2}$ for $f\in L^\phi(\mathbb R^d)$.
\end{lem}
\begin{proof}

In view of Proposition~\ref{prop:decomposition}(2) and linearity of $P_t$, it suffices to consider $p\in (1,\infty)$ and $f\in L^p(\Rd)$. In this case, the smoothness in the spatial variable was proved in \cite{MR3556449} just above Example 1 therein, so we consider the time derivative.
Denote $q=p/(p-1)$. If $g\in L^q(\Rd)$ and $x\in \Rd$, then $L^p(\Rd)\ni h\mapsto h*g(x)\in \R$ is a continuous linear functional. In particular, this applies  to $g:=p_s$ for each $s>0$.
Of course, $P_t$ is a bounded linear operator, in fact, a contraction on $L^p(\Rd)$.
Furthermore, the operator-valued map $(0,\infty)\ni t\mapsto P_{t}$ extends to a sector in $\mathbb C$ and is holomorphic there in the operator norm; see 
\cite[Theorem III.1, p. 67]{MR252961}.
Moreover, any strongly holomorphic function is weakly holomorphic \cite[p. 80]{MR1157815}.
With this in mind, we let $t>s>0$ and $n=0,1,\ldots$, and we see that
\begin{align*}
\frac{{\rm d}^n}{{\rm d}t^n} \left(P_t f(x)\right)=
\frac{{\rm d}^n}{{\rm d}t^n} \left((P_{t-s} f)*p_s(x)\right)=
\left(\left(\frac{{\rm d}^n}{{\rm d}t^n} P_{t-s}\right) f\right)*p_s(x)
\end{align*}
exists and is finite. Thus, $(0,\infty)\ni t\mapsto P_t f(x)$ is smooth.
\end{proof}

We note in passing that $p_t(x,\cdot)\in L^\phi(\R^d)$ for any $t>0,$ $x\in\mathbb R^d$, and moderate Young function $\phi$. Indeed, $\phi(\lambda)/\lambda$ is non-decreasing in $\lambda \ge 0$, so \begin{equation}\label{eq:p_tintegrability} \int_{\mathbb R^d} \phi(p_t(x,y)){\rm d}y = \int_{\mathbb R^d} \frac{\phi(p_t(x,y))}{p_t(x,y)} p_t(x,y){\rm d}y \le \frac{\phi(p_t(0))}{p_t(0)} \int_{\mathbb R^d} p_t(x,y){\rm d}y < \infty. \end{equation}

For later use, we recall that the process $Z$ can be decomposed as
\begin{equation}\label{eq:decompose}
Z_t=D_t+J_t, \quad t\geq 0,
\end{equation}
where $D$ is the diffusive part and $J$ is the jump part.
Furthermore, the diffusive (continuous) part can be written as $D_t = \mathcal R W_t$, where $W$ is a $d$-dimensional Brownian motion independent of $J$. For instance, if $\nu=0$ and $Q={\rm I}_d$, the identity matrix, then $Z$ is the Brownian motion.

\subsubsection{Parabolic martingale}
The
L\'evy
process
$Z$
is a semimartingale \cite[Section I.4]{MR2273672}.
We also consider the so-called \textit{parabolic martingale} $M$, defined as
\begin{equation}\label{e.dM}
M_t:=
P_{T-t}f(Z_t), \quad 0\le t \le T,
\end{equation}
where $f \in L^1(\mathbb R^d)$ and $T \in (0,\infty)$.
\begin{lem}\label{lem:mart}
If $f \in L^\phi(\Rd)$, then $M$ is a $\phi-$integrable martingale under $\mathbb P_x$, for each $x \in \Rd$.
\end{lem}
 \begin{proof} Let $x \in \mathbb R^d$. From Young's inequality, $f \in L^\phi(\R^d)$ and \eqref{eq:p_tintegrability} with $\phi^*$, it follows that \[\mathbb E_x |f(Z_T)|\le \int \phi(f(y)){\rm d}y + \int \phi^* (p_T(x,y)) {\rm d}y <\infty.\] 
By the Markov property,
\[
\mathbb E_x[ f(Z_T) | \mathcal F_t] = \mathbb P_{T-t}f(Z_t)=M_t,\quad t\in [0,T]. \]
Thus the parabolic martingale is closed by $f(Z_T)$.
Furthermore, by the conditional Jensen inequality and \eqref{e.exfzt}, for $t\in [0,T]$ we get
\begin{align*}
\mathbb E_x\phi(M_t) & =\mathbb E_x\phi(\mathbb E_x[ f(Z_T) | \mathcal F_t]) \le
\mathbb E_x(\mathbb E_x[ \phi(f(Z_T)) | \mathcal F_t]) =\mathbb E_x\phi(f(Z_T))<\infty.
\end{align*}
\end{proof}
Recall that for every (fixed) $t\ge0$, $Z_t=Z_{t-}$ a.s. Here is the limiting behavior of $M$.
 \begin{lem}\label{l.nc}
For every $x\in \Rd$,  $\mathbb P_x$-a.s. we have $M_t\to M_T$ as $t\to T$.
\end{lem}
\begin{proof}
This follows from the Martingale Convergence Theorem \cite[Chapter I, Theorem 9]{MR2273672} since $M$ is closed by $f(Z_T)$, which equals $f(Z_{T-})$ a.s., so it is
measurable with respect to   $\mathcal F_{T-}^{Z}$, the $\sigma$-algebra generated by the union of $\mathcal F_t^Z$, $t\in [0,T)$; see Chung \cite[Theorem 10 of Section 1.3]{MR648601}.
\end{proof}
The next result, concerning general symmetric Markov processes, was proved for $\phi(t)=t^p$, $p\in (1,\infty)$, by Kim \cite[Proposition 2.3]{MR3558361}. The proof applies verbatim
to moderate Young functions $\phi$, once we use an extension of  Doob's inequality
given in Corollary \ref{cor:doob-cont}.
\begin{prop}\label{prop:max-Orlicz} Assume that $\phi$ is a $C^1$ moderate Young function with Simonenko indices $1<d_\phi\leq D_\phi<\infty$.  Let $C_\phi = \left(d_\phi/(d_\phi - 1)\right)^{D_\phi}$. Then,
\[
\int_{\mathbb R^d}
\phi\left(\sup_{t\geq0}|P_tf(x)|\right)\,{\rm d}x
\leq C_\phi\int_{\mathbb R^d}\phi(f(x))\,{\rm d}x<\infty,
\qquad f\in L^\phi(\Rd).
\]
\end{prop}
We also have a Stein-type inequality for the Luxemburg norm,
\begin{equation}\label{eq:orl}
\|\sup_{t\geq 0} |P_tf|\|_\phi \le C_\phi\|f\|_\phi.
\end{equation}
Indeed, we can assume $\|f\|_\phi > 0$ and apply Proposition \ref{prop:max-Orlicz} with $\widetilde{f}:=f/\|f\|_\phi$, to obtain
\[\int_{\mathbb R^d}\phi(\sup_{t\geq 0} |P_t \widetilde {f}(x)|)\,{\rm d}x \le C_\phi
\int_{\mathbb R^d} \phi(\widetilde{f}(x)){\rm d}x
=C_\phi
\int_{\mathbb R^d} \phi\left(f(x)/\|f\|_\phi\right){\rm d}x \le C_\phi, \]
by \eqref{eq:orlicz2}. Since $C_\phi \ge 1$, by the convexity of $\phi$ and $\phi(0)=0,$
\[ \int_{\mathbb R^d}\phi\left(\sup_{t\geq 0} |P_t \widetilde {f}(x)|/C_\phi\right)\,{\rm d}x \le\int_{\mathbb R^d}\phi(\sup_{t\geq 0} |P_t \widetilde {f}(x)|)\,{\rm d}x/C_\phi   \le 1,\]
so \eqref{eq:orl} follows the definition of $\|\cdot\|_\phi$.

 \begin{lem}\label{lem:lem2}
If $\phi$ is a $C^1$ moderate Young function and $f \in L^\phi(\Rd)$, then
\begin{equation}\label{eq:lem2}
\lim_{t\to\infty}\int_{\mathbb R^d} \phi(P_tf(x)) {\rm d}x = 0.
\end{equation}
\end{lem}
\begin{proof} By Jensen's inequality and \eqref{e.exfzt}, for $0< t\to \infty$, we get
 \begin{eqnarray*}
 \phi\left(P_tf(x)\right) \le
\mathbb E_x\phi(f(Z_t)) \le p_t(0) \int_{\Rd}\phi(f(y)){\rm d}y  \to 0.
 \end{eqnarray*}
 The result follows from Proposition \ref{prop:max-Orlicz} and the Dominated Convergence Theorem.
 \end{proof}

\subsubsection{The parabolic Hardy--Stein identity} \label{ss.HS}
As before, let $(P_t)_{t\geq 0}$ be the semigroup of the process $Z,$ as given by \eqref{eq:semigroup-def}. Here is the main result of Subsection \ref{s.HS}.
\begin{thm}\label{th:Hardy-general}
Let  $\phi$ be a moderate Young function with absolutely continuous derivative. Under the assumptions of Subsection~\ref{sec:Levy},
for every $f\in L^\phi(\Rd)$ we have
\begin{align}\label{eq:Hardy-general}
\int_{\Rd}\phi(f(x)){\rm d}x = & \frac12\int_0^\infty \int_{\Rd}\phi''(P_tf(y))\left\langle  Q\nabla P_t f(y),\nabla P_t f(y)\right\rangle{\rm d}y {\rm d}t
\nonumber
\\& + \int_0^\infty\int_{\Rd}\int_{\Rd}F_\phi(P_tf(y),P_tf(y+z))\nu({\rm d}z){\rm d}y{\rm d}t.
\end{align}
\end{thm}
\begin{proof}
Fix $f\in L^\phi(\Rd)$ and  $T\in(0,\infty)$,  then consider the parabolic martingale $M=(M_t)_ {0\leq t\leq T}$, given by \eqref{e.dM}.
 After some preparation,
 we will apply  Theorem \ref{prop:integration}
 and Tonelli's theorem to this martingale.
We first calculate  the $\phi$-variation of $M.$  
Let
 \[G(t,y):=
 P_{T-t}f(y),\quad 0\leq t<T,\ y\in \Rd.\]
Thus,  $M_t=G(t,Z_t)$ when $0\le t<T$.
By the discussion following \eqref{eq:HaWi}, in particular Lemma \ref{l.rPtf}, $G(t,x)$ is $C^{1,2}((0,T) \times \mathbb R^d)$, so 
It\^o's formula for semimartingales \cite[Corollary 14.2.9]{MR3443368} can be applied. It yields, for $t \in [0,T)$,
\begin{align} \label{eq:ito-gen}
M_t-M_0 = &\int_0^t\frac{\partial G}{\partial s}(s, Z_{s-})\,{\rm d}s + \int_{0^+}^t \nabla_x G(s, Z_{s-}){\rm d}Z_s
\nonumber
\\
&+ \frac{1}{2}\int_{0^+}^t \sum_{i,j=1}^d \frac{\partial^2 G}{\partial x_i\partial x_j}(s, Z_{s-})\,{\rm d}[Z_i,Z_j]_s^c
\nonumber
\\
& + \sum_{0 < s\leq t,\, \Delta Z_s\neq 0} \!\!\!\!\!\!
G(s,Z_s)-G(s,Z_{s-})-
\left\langle\nabla_x G(s,Z_{s-}), (Z_s-Z_{s-})\right\rangle.
\end{align}

We now apply the formula \eqref{eq:itm-smooth-2}.
By Lemma~\ref{l.nc}, for every
$x\in\Rd$ one has $M_{T-}=M_T,$ $\mathbb P_x$-a.s.  Hence,
$F_\phi(M_{T-},M_T)=0$, $\mathbb P_x$-a.s.
According to \eqref{eq:decompose} and \eqref{eq:ito-gen}, the only term that can contribute to the continuous part of the quadratic variation of $M$ is the stochastic integral
\begin{align*}I_t& := \int_{0+}^t \nabla_x G(s, Z_{s-}){\rm d}Z_s = \int_{0}^t \nabla_x G(s, Z_{s-}) {\rm d}D_s
+\int_{0+}^t \nabla_x G(s, Z_{s-}) {\rm d}J_s
\\
&=: I_t^{(1)}+I_t^{(2)}, \quad 0 \le t < T.
\end{align*}
Since  $J$ is a pure jump process, $[I^{(2)}]^c=0$ and $[I^{(1)},I^{(2)}]^c=0$,  see \eqref{eq:covar-jump}. It follows that
\[
[I]^c_t
=
\left[
\int_0^\cdot \nabla_xG(s,Z_{s-})\,{\rm d}D_s
\right]_t
=
\sum_{i,j=1}^d
\int_0^t
\frac{\partial G}{\partial x_i}(s,Z_{s-})
\frac{\partial G}{\partial x_j}(s,Z_{s-})
\,{\rm d}[D^i,D^j]_s.
\]
It remains to compute the quadratic covariations $[D^i,D^j]$.  Since $D_t = \mathcal R W_t$, it is elementary to see that
\([D^i,D^j]_s= \left(\sum_{k=1}^d r_{ik} r_{kj} \right) s = q_{ij}s\).
 Thus, for every $0 \le t < T,$
\begin{align*}
[M]_t^c = [I]^c_t &=  \int_0^t \langle
Q\nabla_x P_{T-s}f(Z_{s-}),\nabla_x P_{T-s}f(Z_{s-}) \rangle {\rm d}s.
\end{align*}
By Theorem \ref{t.itm-smooth},
\begin{align}\nonumber
V^\phi(M)_T = \phi(M_0) &+ \frac{1}{2}\int_0^T \phi''(M_{s-})\langle Q\nabla_xP_{T-s}f(Z_{s-}),\nabla_xP_{T-s}f(Z_{s-})\rangle{\rm d}s\\
&+ \sum_{0<s<T,\Delta M_s\neq 0} F_\phi(M_{s-},M_s).
\label{eq:var-M}
\end{align}
We will integrate \eqref{eq:var-M} first with respect to $\mathbb P_x$ and then with respect to ${\rm d}x$, the Lebesgue measure.
For the left-hand side of \eqref{eq:var-M}, from Theorem \ref{prop:integration} and Tonelli's theorem,
\begin{equation}\label{eq:iso-M}
\int_{\mathbb R^d} \mathbb E_xV^\phi(M)_T{\rm d}x = \int_{\mathbb R^d} \mathbb E_x\phi(M_T){\rm d}x =\int_{\mathbb R^d}\phi(f(y)){\rm d}y.
\end{equation}
We then consider the terms on the right-hand side of \eqref{eq:var-M}.
Since $p_s(x,y)$ is the  density of both $Z_s$ and $Z_{s-}$ under $\mathbb P_x$, by \eqref{e.sk}, 
we get 
\begin{align}
 \label{eq:n11} \nonumber
 & \int_{\Rd} \mathbb E_x\int_0^T  \phi''(M_{s-}){\rm d}[M]_s^c{\rm d}x  \\ \nonumber
&= \int_0^T\int_{\mathbb R^d}\int_{\mathbb R^d} \phi''(P_{T-s}f(y))\left\langle Q\nabla_x P_{T-s}f(y),\nabla_x P_{T-s}f(y)\right\rangle  p_s(x,y){\rm d}y{\rm d}x{\rm d}s
\\
 &=  \int_0^T \int_{\mathbb R^d} \phi''(P_{t}f(y))\left\langle Q\nabla_x P_{t}f(y),\nabla_x P_{t}f(y)\right\rangle {\rm d}y{\rm d}t.
\end{align}
We then consider $\sum_{0<s< T,\,\Delta M_s\neq 0} F_\phi(M_{s-},M_s)$. Since the jumps of $M$ can only occur at the jump times of $Z$,
it follows that
\begin{align*}
\sum_{0<s< T,\,\Delta M_s\neq 0} F_\phi(M_{s-},M_s) = \sum_{0<s< T,\,\Delta Z_s\neq 0} F_\phi(M_{s-},M_s)=\sum_{0<s< \infty} H_T^\phi(s,Z_{s-},Z_s),
\end{align*}
where
\[H_T^\phi(s,x,y) :=\mathbf 1_{[0,T)}(s) F_\phi(P_{T-s}f(x),P_{T-s}f(y)).\]
This is a nonnegative (Borel) function so by the L\'{e}vy system of $Z$, see, e.g., Bogdan, Rosiński, Serafin and Wojciechowski \cite[Lemma 4.1]{MR3737628},
\begin{align}\nonumber
\mathbb E_x \!\!\! \sum_{0<s< T,\,\Delta M_s\neq 0} \!\!\! F_\phi(M_{s-},M_s)& = \mathbb E_x \int_0^\infty \int_{\mathbb R^d} H_T^\phi(s, Z_s,Z_s+z)\nu({\rm d}z){\rm d}s\\ \nonumber
&=  \mathbb E_x \int_0^T\int_{\mathbb R^d} F_\phi(P_{T-s}f(Z_s), P_{T-s}f(Z_s+z))\nu({\rm d}z) {\rm d}s\\
&=  \int_0^T \int_{\mathbb R^d} \int_{\mathbb R^d} F_\phi(P_{T-s}f(y),P_{T-s}f(y+z))\nu({\rm d}z) p_{s}(x,y){\rm d}y{\rm d}s.
\label{eq:nlk}
\end{align}
Here $x \in \Rd$ is arbitrary. Integrating with respect to ${\rm d}x$ and using Tonelli's theorem and \eqref{e.sk}, we get
\begin{align}\label{eq:n12}\nonumber
\qquad \qquad & \int_{\mathbb R^d}\mathbb E_x\!\!\!\sum_{0<s< T,\,\Delta M_s\neq 0} \!\!\! F_\phi(M_{s-},M_s){\rm d} x\\
\nonumber
& =
 \int_{\mathbb R^d}\int_0^T \int_{\mathbb R^d} \int_{\mathbb R^d} F_\phi(P_{T-s}f(y),P_{T-s}f(y+z))\nu({\rm d}z) p_{s}(x,y){\rm d}y{\rm d}s{\rm d}x\\ \nonumber
 &= \int_0^T \int_{\mathbb R^d} \int_{\mathbb R^d} F_\phi(P_{T-s}f(y),P_{T-s}f(y+z))\nu({\rm d}z){\rm d}y{\rm d}s\\
 &= \int_0^T \int_{\mathbb R^d} \int_{\mathbb R^d} F_\phi(P_{t}f(y),P_{t}f(y+z))\nu({\rm d}z){\rm d}y{\rm d}t.
  \end{align}

Considering \eqref{eq:var-M}, \eqref{eq:iso-M},  \eqref{eq:n11} and \eqref{eq:nlk}, 
we get
\begin{align*}
\int_{\mathbb R^d}\phi(f(x)){\rm d}x  &=
  \frac12\int_0^T \int_{\mathbb R^d} \phi''(P_{t}f(y))\left\langle Q\nabla_x P_{t}f(y),\nabla_x P_{t}f(y)\right\rangle {\rm d}y{\rm d}t\\
&\phantom{ppp}+
 \int_0^T \int_{\mathbb R^d} \int_{\mathbb R^d} F_\phi(P_{t}f(y),P_{t}f(y+z))\nu({\rm d}z){\rm d}y{\rm d}t+\int_{\mathbb R^d} \phi(P_Tf(x)){\rm d}x. \end{align*}

We let $T\to\infty$.
By
Lemma \ref{lem:lem2},
$\int_{\mathbb R^d} \phi(P_Tf(x)){\rm d}x\to 0$.  All the integrands are nonnegative, so the Monotone Convergence Theorem yields \eqref{eq:Hardy-general}.
\end{proof}

\subsection{Elliptic Hardy--Stein identity}

In this subsection, we prove Hardy--Stein identities for harmonic functions of L\'evy processes. The results generalize \cite[Section 3]{MR4589708} by allowing rather general convex functions $\phi$ and more general processes. 
Let $Z$ be the  L\'{e}vy process from Subsection \ref{sec:Levy}. 
For technical reasons that will become apparent later on, we \textit{additionally} assume that 
the L\'evy measure $\nu$ is \textit{unimodal} \cite{MR3729529}. More precisely, we require $\nu({\rm d}x)=\nu(|x|){\rm d}x$, with nonincreasing \textit{profile} $\nu:(0,\infty)\mapsto (0,\infty)$ and $Q=\sigma^2 I$ with $\sigma \ge 0$. Then the  L\'evy-Khinchine exponent of $Z$ is 
	\begin{equation*}
		\psi(\xi) = \frac{\sigma^2}{2} |\xi|^2 + \int_{\Rd} \left(1-\cos\langle\xi, x\rangle\right)\nu(|x|){\rm d}x,\quad \xi\in \Rd.
	\end{equation*}
We keep the Hartman-Wintner condition \eqref{eq:HaWi}; it suffices, e.g., if $\sigma>0$.

\subsubsection{Harmonic functions} \label{ss:Harmonic}
For an open set $U\subset \Rd$, we define 
\begin{equation}\label{e.dtU}
\tau_U:=\inf\{t\geq 0: Z_t\notin U\},
\end{equation}
the time of the first exit of $Z$ from $U$. Let  $u:\Rd\to\R$ and assume that
\begin{equation}\label{eq:reg-harmonic}
u(x)=\mathbb E_x [u(Z_{\tau_U}); \tau_U<\infty],\quad x\in U,
\end{equation}
and that the expectations are absolutely convergent:
we say that such $u$ is \textit{regular harmonic}
in $U$ with respect to $Z$. Of course,  the equality in \eqref{eq:reg-harmonic} always holds for $x\in U^c$. If $U$ is bounded, then $\tau_U<\infty$ $a.s.$ with respect to every $\mathbb P_x$, $x\in \Rd$.
Guided by \cite[Remark~4.4]{MR4088505}, we see that
\begin{equation}\label{eq.dme}
X_t:= u(Z_{\tau_U\wedge t}),\quad t\ge 0,
\end{equation}
is a martingale closed by $u(Z_{\tau_U}){\bf 1}_{\tau_U<\infty}$ with respect to each $\mathbb P_x$. Let us briefly explain that it is a c\`adl\`ag martingale.
Due to $u(x)=\mathbb E_x u_+(Z_{\tau_U}) -\mathbb E_x u_-(Z_{\tau_U})$, $x\in \Rd$, it suffices to consider $u \ge 0$. By inspection, the \textit{stopped process} $t\mapsto Z_{\tau_U \wedge t}$ is a Hunt process. By the above martingale property, 
$u$ is \textit{excessive} for the stopped process; see \cite[II.3]{MR850715} or \cite[Chapter 2]{MR648601} for the definition.
By the composition theorem \cite[IV, Corollary 4.13]{MR850715}, $u(Z_{\tau_U\wedge t})$ is right continuous ($a.s.$). On the other hand, by Doob's regularization theorem  for martingales \cite[Chapter II, Theorem 2.9]{MR1725357}, $X_t$ has a c\`adl\`ag \textit{modification}. 
Since $X_t$ is already right continuous, it is actually c\`adl\`ag.

We fix an arbitrary (nonempty) bounded open set $D\subset \Rd$. If \eqref{eq:reg-harmonic} holds for all open bounded (relatively compact) sets $U\subset \overline{U}\subset D$, then we say that $u$ is {\em harmonic} in $D.$

Here is a Weyl-type lemma for unimodal L\'evy operators, an extension of \cite[Theorem~1.7]{MR3729529} and  \cite[Theorem~4.9 and 5.7]{MR4088505}. Note that \cite[Theorem~1.7]{MR3729529} only dealt with bounded harmonic functions and  \cite[Theorem~4.9 and 5.7]{MR4088505} only dealt with purely nonlocal operators.
Our variant makes no additional growth or integrability assumptions on harmonic functions $u$
 and allows for $\sigma>0$. We thank Mateusz Kwaśnicki and Artur Rutkowski for discussions on this extension.
\begin{lem} \label{lem:aux1}
	Assume that for every $r_0> 0$, there is a constant $C(r_0)$ such that $|\nu(r)| \lesssim C(r_0)\nu(r+1)$ for $r>r_0$ and $\nu^{(n)}$ is continuous on $(0,\infty)$, with $|\nu^{(n)}(r)| \lesssim C(r_0)\nu(r)$, for $r>r_0$ and $n=1,\ldots, N$. If $u$ is harmonic in $D$,
then $u\in C^N(D)$.
\end{lem}
\begin{proof}
It suffices to prove that $u\in C^N(B_r)$ if $0\in D$ and $r>0$ is small enough. Here, $B_r:=\{y\in \Rd:\; |y|<r\}$. To this end, recall that $\nu>0$. For $r>0$ we consider  the harmonic measure at the center of the ball $B_r$,
\[P_r(A):=\mathbb P_0(X_{\tau_{B_r}}\in A), \quad A\subset \Rd.
\]
On $(\bar{B_r})^c$, $P_r(\cdot)$ has the density function 
\[P_r(z):=\int_{B_r} G_r(0,y)\nu(z-y)dy\ge \nu(|z|+r) \int_{B_r} G_r(0,y)dy,\quad |z|>r.
\]
Here $G_r$ is the Green function of $B_r$ and $\int_{B_r} G_r(0,y)dy=\mathbb E_0 \tau_{B_r}\in (0,\infty)$, see \cite[Section 2.1]{MR3729529}.
As before, we may assume that $u\ge 0$.
If $x\in D$ and $0<r<{\rm dist}(x,D^c)$, then 
\begin{align*}
\infty>u(x)&=\int_{\Rd}u(x+y) P_r(dy)\ge \int_{|z|>r}u(x+z) P_r(z)dz\\
&\ge\mathbb E_0 \tau_{B_r} \int_{|z|>r} u(x+z)\nu(|z|+r)dz.
\end{align*}
Considering two such points $x$ and small enough $r>0$, we see that $u$ is locally integrable on $\Rd$. In fact, $\int_{\Rd} u(z)\;  \nu(1\wedge|z|) dz<\infty$. At this moment, we may just follow the proof of \cite[Theorem~1.7]{MR3729529}. In effect, if $u$ is regular harmonic in $B_{(2k+1)r}\subset D$ (the assumption was glossed over in \cite{MR3729529} and  \cite{MR4088505}), then $u\in C^N(B_r)$, where $k$ is a suitable integer depending on $N$ and $d$, as specified in the cited proof. 
\end{proof}

\subsubsection{Local It\^o formula}\label{sec:itolocal}
For the convenience of the reader, we state and verify the following result; see also \cite[p. 334]{MR745449}.
\begin{lem}\label{eq:itolocal}
If 
$(Z_t, t\ge 0)$ is a semimartingale with values in $\Rd$, $u:\Rd\to \R$,  $u\in C^2(D)$, 
$U$ is an open, relatively compact subset of $D$, and $\tau_U$ is given by \eqref{e.dtU},
then 
$(u(Z_t^{\tau_U}), t\ge 0)$ is a semimartingale  with the representation
\begin{align}\label{eq:ito-local}
u(Z_t^{\tau_U}) = & u(Z_0)+\int_{0+}^{t\wedge \tau_U} \nabla u(Z_{s-}){\rm d}Z_s +\frac{1}{2}\sum_{i,j=1}^d\int_{0+}^{t\wedge \tau_U}
\frac{\partial^2 u}{\partial x_i\partial x_j}(Z_{s-}){\rm d}[Z^i,Z^j]^c_s \nonumber\\
&+ \sum_{0 < s\leq t\wedge \tau_U, \Delta Z_s\neq 0} \left(
u(Z_s)-u(Z_{s-}) -\nabla u(Z_{s-})\cdot(Z_s-Z_{s-})\right).
\end{align}
In particular,
\begin{equation}\label{eq:quad-var-cont}
[u(Z^{\tau_U})]^c_t = \int_0^{t\wedge\tau_U} \sum_{i,j=1}^d \frac{\partial u}{\partial x_i} \frac{\partial u}{\partial x_j} (Z_{s-}) {\rm d}[Z_i,Z_j]^c_s. 
\end{equation}
\end{lem}
\begin{proof} The equality is trivial if $Z_0\in U^c$. Furthermore, if $u\in C^2(\Rd)$ then \eqref{eq:ito-local} follows from the It\^o formula \cite[Theorem II.32]{MR2273672} by considering \textit{time} $t\wedge \tau_U$. 
Else, we fix $z_0\in U$ and define
\[\widetilde{Z}_t:= Z_t\mathbf 1_{\{t<\tau_U\}} + z_0 \mathbf 1_{\{t\geq\tau_U\}}.\] 
Of course, the process $Z$ \textit{stopped} at $\tau_U$ has a similar decomposition:
$$Z_t^{\tau_U}=Z_t\mathbf 1_{\{t<\tau_U\}} + Z_{\tau_U} \mathbf 1_{\{t\geq\tau_U\}}.$$
Both processes are semimartingales and so is their difference,
$$R_t:=Z_t^{\tau_U}-\widetilde{Z}_t=(Z_{\tau_U}-z_0) \mathbf 1_{\{t\geq\tau_U\}}.$$
We extend $u$ from $\overline U$ to a function $\widetilde u\in C^2(\Rd)$. 
Thus, $\widetilde u(z)=u(z)$ for $z\in \overline  U$, $\widetilde Z_{s}=Z_{s}$ for $s<\tau_U$ and $\widetilde Z_{s-}=Z_{s-} \in \overline U$ for $s\le\tau_U$.
By the above discussion and \eqref{e.si},
\begin{align*}
\widetilde u(\widetilde Z_{t\wedge \tau_U}) = & u(Z_0)+\int_{0+}^{t\wedge \tau_U} \nabla u(Z_{s-}){\rm d}Z_s-
\int_{0+}^{t\wedge \tau_U} \nabla u(Z_{s-}){\rm d}R_s\\
&+\frac{1}{2}\sum_{i,j=1}^d\int_{0+}^{t\wedge \tau_U}
\frac{\partial^2 u}{\partial x_i\partial x_j}(Z_{s-}){\rm d}[Z^i,Z^j]^c_s 
\nonumber\\
&+ \sum_{0 < s\leq t\wedge \tau_U, \Delta Z_s\neq 0} \left(
u(Z_s)-u(Z_{s-}) -\nabla u(Z_{s-})\cdot(Z_s-Z_{s-})\right)\\
&+{\bf 1}_{t\ge \tau_U}\left(\widetilde u(\widetilde Z_{\tau_U})-u(Z_{\tau_U})-\nabla u(Z_{\tau_U-})\cdot(\widetilde Z_{\tau_U}-Z_{\tau_U})\right).
\end{align*}
Clearly, $\widetilde Z_{\tau_U}=z_0$, $\widetilde u(z_0)=u(z_0)$, $\widetilde u(\widetilde Z_{t\wedge \tau_U})=
u(Z_t^{\tau_U})+{\bf 1}_{t\ge \tau_U} \left(u(z_0)-u(Z_{\tau_U})\right)$,
and $\int_0^{t\wedge \tau_U} \nabla u(Z_{s-}){\rm d}R_s={\bf 1}_{t\ge \tau_U}\nabla u(Z_{\tau_U-})(Z_{\tau_U}-z_0)$, so
\eqref{eq:ito-local} follows by inspection. To obtain \eqref{eq:quad-var-cont}, we note that with $t \wedge \tau_U$ replaced by $t$, the continuous part of the quadratic variation of the right-hand side of \eqref{eq:ito-local} is precisely the right-hand side of \eqref{eq:quad-var-cont} with $t \wedge \tau_U$ replaced by $t$. Hence, the result follows by \eqref{e.gwiazdka}.
\end{proof}

We now collect the assumptions used in the next two results. We assume that
\[
\nu({\rm d}z)=\nu(|z|)\,{\rm d}z
\quad\text{and}\quad
Q=\sigma^2 I,\qquad \sigma\geq0,
\]
where the profile $\nu:(0,\infty)\to(0,\infty)$ is nonincreasing.
We retain the Hartman--Wintner condition \eqref{eq:HaWi} and assume
that $\phi'$ is absolutely continuous. We also assume that the profile
$\nu$ satisfies the hypotheses of Lemma~\ref{lem:aux1} with $N=2$.
The following proposition extends part of the results of
\cite{MR3737628}.
Recall that the Green operator of  $D$  for $Z$ is defined by
$$
\int_D G_D(x,y) f(y){\rm d}y:=\mathbb E_x \int_0^{\tau_D}f(Z_s){\rm d}s,\quad x\in \Rd,\quad f\geq 0.
$$

\begin{prop}\label{p.piD} Let 
$D\subset \Rd$ be open and   let
$u:\Rd\to\R$ be regular harmonic in $D$ with respect to $Z$.
If $x\in D$ and $\mathbb E_x\phi(u(Z_{\tau_D}))<\infty$, then 
   \begin{align} \label{eq:sss} 
\mathbb E_x\phi(u(Z_{\tau_D})) = & \phi(u(x))  \\
&+ \int_D G_D(x,y)\left[ \frac{\sigma^2}{2} \phi''(u(y))|\nabla u(y)|^2 
+ \int_{\Rd} F_\phi(u(y), u(y+z))\nu(z){\rm d}z\right]{\rm d}y.\nonumber
  \end{align}  
\end{prop}

\begin{proof}
Let $\emptyset \neq U \subset \overline {U}\subset D$ and let $X$ be given by \eqref{eq.dme}. Lemma \ref{lem:aux1} yields that $u\in C^2(D),$ so we read off the continuous part of the quadratic variation of $X$ from \eqref{eq:quad-var-cont}, which in this case equals to $\sigma^2\int_0^{t\wedge \tau_U} |\nabla u(Z_{s-})|^2{\rm d}s. $ Consequently, by
 Theorem~\ref{t.itm-smooth},
the $\phi$-variation of the martingale $X_t$ in \eqref{eq.dme}
is
   \begin{align}\label{eq:itm-smoothD}
   \nonumber V_{\tau_U\wedge t}^\phi(X)  &= \phi(X_0)+  \frac{\sigma^2}{2}\int_0^{\tau_U\wedge t} \phi''(u(Z_{s-}))|\nabla u(Z_s)|^2 {\rm d}s\\ &+ \sum_{0<s\le \tau_U\wedge t, \Delta Z_s\neq 0} F_\phi(u(Z_{s-}),u(Z_s)).
\end{align}

By Jensen's inequality, we have $\mathbb E_x \phi(X_t)<
\infty$, so by Theorem~\ref{prop:integration} and the L\'evy system \cite[Lemma 4.7]{MR3737628} (note that $g_s := 1_{s \le \tau_U \wedge t}$ is predictable),
\begin{align} \label{eq:eliptic_stoped_phi_variation} \nonumber
\mathbb E_x\phi(u(Z_{\tau_U\wedge t}))=\mathbb E_x V_{\tau_U\wedge t} ^\phi(X) & = \phi(u(x)) + \frac{\sigma^2}{2} \mathbb E_x\left[ \int_0^{\tau_U\wedge t} \phi''(u(Z_{s-}))|\nabla u(Z_s)|^2 {\rm d}s\right] \nonumber \\
& + \mathbb E_x \left[\int_0^{\tau_U\wedge t} \int_{\Rd} F_\phi(u(Z_s), u(Z_s+z))\nu(z){\rm d}z{\rm d}s\right].
\end{align}
We now let $t\to\infty.$ The right-hand side of \eqref{eq:eliptic_stoped_phi_variation} is monotone in $t,$ so we can use the Monotone Convergence Theorem.

As for the left-hand side, 
we have pointwise convergence under $\mathbb E_x,$ and the convergence of integrals follows from the Dominated Convergence Theorem, since
\[\mathbb E_x\left[\sup_{t} \phi(X_t)\right] \leq C_\phi \sup_{t}\mathbb E_x\phi(X_t) 
= C_\phi \mathbb E_x\phi(u(Z_{\tau_U})) <\infty. \]
This gives
\begin{align} \label{eq:eliptic_phi_variation} \nonumber
\mathbb E_x\phi(u(Z_{\tau_U})) &= \phi(u(x))
+ \frac{\sigma^2}{2} \mathbb E_x\left[ \int_0^{\tau_U} \phi''(u(Z_{s-}))|\nabla u(Z_s)|^2 {\rm d}s\right] \nonumber \\
&+ \mathbb E_x \left[\int_0^{\tau_U} \int_{\Rd} F_\phi(u(Z_s), u(Z_s+z))\nu(z){\rm d}z{\rm d}s\right].
\end{align}
Using the Green operator, by Tonelli's theorem we obtain the result for $U.$ Then we exhaust $D$ with open sets $U_n$ compactly included in $D,$ and pass to the limit $n\to\infty.$ The right-hand side of \eqref{eq:eliptic_phi_variation} is monotone in $U$. For the left-hand side, we use the martingale convergence (in $L^\phi$) for $u(Z_{ \tau_{U_n}})$.
\end{proof}

If $u$ is harmonic in $D$ (not necessarily regular harmonic), then by monotonicity, we get the following statement.

\begin{cor}
Suppose that $D$ is as above and $u:\Rd\to \R$ is harmonic in $D.$
Then,  for any $x\in D,$
\begin{align*}
    \sup_{U \mbox{\tiny open,}x\in U\subset \overline{U}\subset D} & \mathbb E_x\phi(u(Z_{\tau_U})) = \phi(u(x))\\
    & +
 \int_D G_D(x,y)\left[ \frac{\sigma^2}{2} \phi''(u(y))|\nabla u(y)|^2+ \int_{\Rd} F_\phi(u(y), u(y+z))\nu(z){\rm d}z\right]{\rm d}y. 
\end{align*}
\end{cor}

\subsection{Examples}
Let $\phi(\lambda)=\lambda^p$ with $p\in(1,\infty)$. Taking $Q=0$ in Theorem~\ref{th:Hardy-general}, we recover \cite[Theorem 3.2]{MR3556449}.  On the other hand, taking $\nu=0$ and $B=2 {\rm I}_d$, we recover \cite[(1.4)] {MR4851904}:
\begin{align}
		\label{eq:HSgauss2}
		\int_\Rd |f(x)|^p \,{\rm d}x = p(p-1) \int_0^\infty \int_\Rd |P_t f(x)|^{p-2} |\nabla P_t f(x)|^2 \,{\rm d}x {\rm d}t,
\end{align}
where $f\in L^p(\Rd)$,
	\begin{align*}
		P_t f(x)=\int_\Rd f(y) p_t(x,y) \,{\rm d}t,
		\quad t>0,\,x\in\Rd,
	\end{align*}
	and
	\begin{align*}
		p_t(x,y) = (4\pi t)^{-d/2} {\rm  e}^{-\frac{|y-x|^2}{4t}},
		\quad t>0,\,x,y\in\Rd.
	\end{align*}
Note that $Q$ and $\nu$ are combined within one formula \eqref{eq:Hardy-general}. This is new and difficult to obtain analytically. (The formula is \textit{not} additive in $Q$ and $\nu$ because $P_t$ depends on both.)

Similarly, \eqref{eq:sss} gives rise to separate formulations of the elliptic Hardy--Stein identities for pure-jump or 
diffusion L\'{e}vy processes. Again, the formula for processes with nontrivial  diffusion and jump parts is new.

The class of moderate Young functions with absolutely continuous derivative contains $|\lambda|^p$ and functions with the following growth at infinity:
$$
|\lambda|^p\log^{\gamma_1}(e+|\lambda|),\qquad
|\lambda|^p\log^{\gamma_1}(e+|\lambda|)
\big(\log\log(e^e+|\lambda|)\big)^{\gamma_2},
\quad \text{etc.}
$$
Here $p\in(1,\infty)$ and $\gamma_1,\gamma_2,\ldots\in\mathbb R$. Indeed, let $h$ denote one of the displayed functions. Direct differentiation gives
$$
\lim_{\lambda\to\infty}\frac{\lambda h'(\lambda)}{h(\lambda)}=p,
\qquad
\lim_{\lambda\to\infty}\frac{\lambda^2h''(\lambda)}{h(\lambda)}
=p(p-1).
$$
Hence, for some sufficiently large $R$, the function $h$ is strictly increasing and strictly convex on $[R,\infty)$. We can modify $h$ on $[-R,R]$ so as to obtain a moderate Young function with absolutely continuous derivative. Such Young functions fall within the scope of Theorem~\ref{th:Hardy-general}. For instance, increasing $R$ if necessary, we may define
$$
q:=\frac{Rh'(R)}{h(R)}>1
$$
and
$$
\phi(\lambda):=
\begin{cases}
h(R)(|\lambda|/R)^q, & |\lambda|\leq R,\\
h(|\lambda|), & |\lambda|>R.
\end{cases}
$$
The values and the first derivatives of the two expressions on the right-hand side agree at $\pm R$. Therefore, $\phi$ is a Young function with absolutely continuous derivative. Moreover,
$$
\frac{\lambda\phi'(\lambda)}{\phi(\lambda)}=q,\qquad 0<\lambda<R,
$$
while
$$
\frac{\lambda\phi'(\lambda)}{\phi(\lambda)}\longrightarrow p
\qquad\text{as }\lambda\to\infty.
$$
By choosing $R$ sufficiently large in this construction, we therefore obtain
$$
1<d_\phi\leq D_\phi<\infty.
$$
Thus, by the discussion following \eqref{e.ddDphi}, $\phi$ is moderate. Young functions with such growth at infinity widely appear in the theory of PDEs; see, e.g., \cite{MR4518648} and the references therein.

\appendix

\section{Random variables and martingales in $L^\phi$} \label{a.MM}
We collect here some properties of random variables and processes that are $\phi$-integrable. Since we did not find a comprehensive reference, we provide some proofs.

\begin{lem}\label{lem:A-2}
 Assume that $\phi \in \Delta_2$ is a Young function. Let $X,Y$ be $\phi$-integrable random variables. Then, $\phi'(X)Y \in L^1(\mathbb P)$.

\end{lem}

\begin{proof} Since $\phi \in \Delta_2$, we get $|\phi'(x)| \le D_\phi\phi(x)/|x|, x \in \mathbb R$, where $\phi(x)/|x|$ is  interpreted as $0$ if $x=0,$
Moreover, $\phi(x)|y|/|x| \le \phi(x)$ for $|x| \ge |y|$ and $\phi(x)|y|/|x| \le \phi(y)$ for $|y| \ge |x|$ because $\phi(x)/x$ is nondecreasing on $(0,\infty)$. Hence, $$ |\phi'(x)y| \le D_\phi\frac{\phi(x)}{|x|}|y| \le D_\phi\big( \phi(x) + \phi(y)\big), \quad x,y \in \mathbb R.$$
The desired integrability follows from the $\phi$-integrability of $X$ and $Y$.
\end{proof}

The following result  is a generalization of Doob's $L^p$-inequality, see Alsmeyer and Rösler \cite[Proposition 1.1]{MR2222745} see also K\"{u}hn and Schilling \cite{MR4568704} for a  survey and Osękowski \cite{MR2964297} for specialized results. 

\begin{lem}\label{th:A1}
If $\phi$ is a $C^1$ moderate Young function, then for any martingale $(X_n)_{n \geq 0}$,
\begin{equation}\label{eq:doob-phi}
\mathbb E[\sup_{k \le n} \phi(X_k)] \le C_\phi \mathbb E[\phi(X_n)],\; n  = 0,1,2,\ldots,
\end{equation}
where $C_\phi=\left(d_\phi/(d_\phi-1)\right)^{D_\phi}$ and $d_\phi$ and $D_\phi$ are the Simonenko indices of $\phi$.
\end{lem}
Here is a variant for continuous time martingales, with the same $C_\phi$.
\begin{cor}\label{cor:doob-cont}
If $\phi$ is a $C^1$ moderate Young function, then for every
 c\`{a}dl\`{a}g martingale  $X=(X_t)_{0 \le t <\infty}$,
 \begin{equation}\label{eq:doob-phi-cont}
 \mathbb E[\sup_{0 \le s \le t} \phi(X_s)] \le C_\phi \mathbb E[\phi(X_t)], \quad 0\le t <\infty.
 \end{equation}
\end{cor}
\begin{proof}
Let $t \in [0,\infty)$. Let $\{D_n\}_{n \in \mathbb N_+}$ be an ascending sequence of finite subsets of $[0,t]$ containing $t$ and such that the set $D := \bigcup_{n=1}^\infty D_n$ is dense in $[0,t]$. By the right continuity of $X$, the Monotone Convergence Theorem, Lemma \ref{th:A1} and Jensen's inequality, we get
$$ \mathbb E[\sup_{s \le t}\phi(X_s)] = \mathbb E[\sup_{s \in D } \phi(X_s)] = \lim_{n \to \infty}\mathbb E[\sup_{s \in D_n} \phi(X_s)] \le C_\phi\mathbb E[\phi(X_{t})].\qedhere $$

\end{proof}

We finish this subsection with a simple lemma concerning the continuous embedding of $(L^\phi(\Omega,\mathbb P),\|\cdot\|_\phi)$ into $(L^1(\Omega,\mathbb P),\|\cdot\|_1)$. For completeness, we give a proof.

\begin{lem}\label{le:embedding}
Assume that $(\Omega,\mathcal F,\mathbb P)$ is a probability space and let $\phi$ be a Young function. 
Then there exists $A_\phi$ such that
$\|X\|_1 \le A_\phi \|X\|_\phi$ for all $X \in L^\phi(\Omega,\mathbb P)$.
\end{lem}
\begin{proof}
Due to homogeneity, we may assume that $\|X\|_\phi = 1$; the case $\|X\|_\phi = 0$ is trivial.
Since $\phi(\lambda)$ is convex and non-zero for large $|\lambda|$, there exist constants $a_\phi,b_\phi > 0$ such that $\phi(\lambda) \ge a_\phi \lambda - b_\phi$ for every $\lambda \ge 0$. By this and \eqref{e.dLn},
\[ \|X\|_1 = \int_{\Omega} |X| {\rm d} \mathbb P \le \frac{1}{a_\phi}\int_{\Omega}\left(\phi(|X|) + b_\phi\right){\rm  d}\mathbb P \le \frac{1+b_\phi}{a_\phi} = \frac{1+b_\phi}{a_\phi}\|X\|_\phi.\qedhere \]

\end{proof}

\section{A primer on stochastic integration}\label{sec:stochint}\label{a.si}
The following are excerpts from  \cite{MR2273672}.
\begin{dfn}
 Process $H$ is said to be \textbf{elementary predictable}, if \begin{equation} \label{eq:elem-pred}
 H_t = H_01_{\{0\}}(t)+\sum_{k=1}^{n} H_k 1_{(T_k,T_{k+1}]}(t), \qquad t \in [0,\infty), \end{equation}
  where $n \in \mathbb N_+$, $0= T_1 \le \ldots \le T_{n+1} < \infty$ are stopping times and $H_0,\ldots,H_n$ are a.s. finite random variables such that $H_k$ is $\mathcal F_{T_k}$-measurable, $k=1,2,\ldots$
\end{dfn}
 By $\mathcal S$ we denote the family of all elementary predictable processes.
 The linear space $\mathcal S$  endowed with the topology of uniform convergence on $\Omega \times [0,\infty)$ will be denoted $\mathcal S_u$.

By $\mathbb D$ (respectively $\mathbb L)$ we denote the spaces of càdlàg, adapted processes (respectively, càglàd, adapted processes). For a process $Y$, let $Y^*_t = \sup_{s \le t}|Y_s|$ for $t \ge 0$. We consider the following notion of convergence.

\begin{dfn}
A sequence of càdlàg processes $(X^{(n)})_{n \in \mathbb N_+}$ converges to a process $X$ \textbf{uniformly on compacts in probability} (in short, \textbf{ucp}) if for every $t\in [0,\infty)$ we have $(X^{(n)} - X)^*_t \xrightarrow[n \to \infty]{\mathbb P} 0$. We then write $X^{(n)} \xrightarrow[]{ucp}X$.
\end{dfn}

By $\mathbb D_{ucp},
 \mathbb L_{ucp},
 \mathcal S_{ucp}$ we denote the spaces $\mathbb D,$
 $\mathbb L,$
 $\mathcal S$ (respectively) endowed with \textit{ucp} topology. Let us recall that given a càdlàg, adapted process $X=(X_t)_{t \ge 0}$ and stopping time $\tau$, the stopped process $X^\tau$, defined by $X^\tau_t = X_{\min\{t,\tau\}}$ is also càdlàg and adapted.

\begin{dfn}
For a semimartingale $X$ we define the mapping $J_X : \mathcal S_{ucp} \to \mathbb D_{ucp}$ by
$$ J_X(H) = H_0X_0 + \sum_{k=1}^n H_k(X^{T_{k+1}} - X^{T_k}),$$
for $H \in \mathcal S$ of the form \eqref{eq:elem-pred}.
\end{dfn}
 $J_X$ is a well-defined linear mapping. If $X$ is càdlàg, then $J_X(H) \in \mathbb D$ for every $H \in \mathcal S$, because $J_X(H)$ is a sum of càdlàg processes.
As the space $\mathcal S_{ucp}$ is dense in $\mathbb L_{ucp}$, and for
a semimartingale $X$, the mapping $J_X:\mathcal S_{ucp} \to \mathbb D_{ucp}$ is continuous
\cite[Chapter II, Theorems 10, 11]{MR2273672}, the mapping  $J_X$ extends from $\mathcal S_{ucp}$ to $\mathbb L_{ucp}$. It is called the stochastic integral with respect to $X$.
 Summarizing, the integral
  \[J_X(Y) = (J_X(Y)_t)_{t \ge 0} = \Big(\int_0^t Y_sdX_s\Big)_{t \ge 0} \]
  is well defined for every semimartingale $X$ and every process $Y \in \mathbb L_{ucp}$. If $X$ is a local martingale, then $J_X(Y)$ is a local martingale as well \cite[Chapter III, Theorem 33]{MR2273672}. Of course, $J_X(Y)$ is denoted $Y\cdot X$ in our development above.

A sequence of random partitions $\pi^{(n)} = \{0=\tau_0^{(n)}\le \tau_1^{(n)} < \ldots < \tau_{k_n}^{(n)}\},$ $n=1,2,\ldots$, is said to tend to the identity when $n\to\infty,$  if $\diam(\pi^{(n)}) = \max_{j<n}|\tau_{j+1}^{(n)}-t_j^{(n)}|$ converge to $0$ a.s. and $\tau_{k_n}^{(n)} \to \infty$ a.s.
Here $\tau_j^{(n)}$ are assumed to be stopping times.
The next theorem asserts that the stochastic integral  above can be approximated by discrete integral sums (see \cite[Chapter II, Theorem 29]{MR2273672}).
\begin{thm}\label{th:approx}
Let $X=(X_t)_{t \in (0,\infty)}$ be a semimartingale and let $Y \in \mathbb L$. If
\[\pi^{(n)}=\{0=t_0^{(n)}<t^{(n)}_1<\ldots<t_{k_n}^{(n)}=T\},\quad n=1,2,\ldots,\]
is a sequence of random partitions tending to identity,
then
$$ \sum_{j=0}^{k_{n}-1} Y_{t_j^{(n)}}\big( X^{t_{j+1}^{(n)}} - X^{t_j^{(n)}}\big) \xrightarrow[n \to \infty]{\text{ucp}} J_X(Y).$$
\end{thm}

\begin{thebibliography}{10}

\bibitem{MR2222745}
G.~Alsmeyer and U.~R\"{o}sler.
\newblock Maximal {$\phi$}-inequalities for nonnegative submartingales.
\newblock {\em Teor. Veroyatn. Primen.}, 50(1):162--172, 2005.

\bibitem{MR3495836}
S.-i. Amari.
\newblock {\em Information geometry and its applications}, volume 194 of {\em Applied Mathematical Sciences}.
\newblock Springer, [Tokyo], 2016.

\bibitem{MR3701408}
N.~Ay, J.~Jost, H.~V. L\^{e}, and L.~Schwachh\"{o}fer.
\newblock {\em Information geometry}, volume~64 of {\em Ergebnisse der Mathematik und ihrer Grenzgebiete. 3. Folge. A Series of Modern Surveys in Mathematics}.
\newblock Springer, Cham, 2017.

\bibitem{BanerjeeGuoWang2005}
A.~Banerjee, X.~Guo, and H.~Wang.
\newblock On the optimality of conditional expectation as a
{B}regman predictor.
\newblock {\em IEEE Trans. Inform. Theory}, 51(7):2664--2669, 2005.

\bibitem{BanerjeeMeruguDhillonGhosh2005}
A.~Banerjee, S.~Merugu, I.~S. Dhillon, and J.~Ghosh.
\newblock Clustering with {B}regman divergences.
\newblock {\em J. Mach. Learn. Res.}, 6:1705--1749, 2005.

\bibitem{MR3556449}
R.~Ba\~{n}uelos, K.~Bogdan, and T.~Luks.
\newblock Hardy-{S}tein identities and square functions for semigroups.
\newblock {\em J. Lond. Math. Soc. (2)}, 94(2):462--478, 2016.

\bibitem{MR3994925}
R.~Ba\~{n}uelos and D.~Kim.
\newblock Hardy-{S}tein identity for non-symmetric {L}\'{e}vy processes and {F}ourier multipliers.
\newblock {\em J. Math. Anal. Appl.}, 480(1):123383, 20, 2019.

\bibitem{MR850715}
J.~Bliedtner and W.~Hansen.
\newblock {\em Potential theory}.
\newblock Universitext. Springer-Verlag, Berlin, 1986.
\newblock An analytic and probabilistic approach to balayage.

\bibitem{MR4600287}
K.~Bogdan, D.~Fafu\l{}a, and A.~Rutkowski.
\newblock The {D}ouglas formula in {$L^p$}.
\newblock {\em NoDEA Nonlinear Differential Equations Appl.}, 30(4):Paper No. 55, 22, 2023.

\bibitem{MR4088505}
K.~Bogdan, T.~Grzywny, K.~Pietruska-Pa\l{}uba, and A.~Rutkowski.
\newblock Extension and trace for nonlocal operators.
\newblock {\em J. Math. Pures Appl. (9)}, 137:33--69, 2020.

\bibitem{MR4589708}
K.~Bogdan, T.~Grzywny, K.~Pietruska-Pa\l{}uba, and A.~Rutkowski.
\newblock Nonlinear nonlocal {D}ouglas identity.
\newblock {\em Calc. Var. Partial Differential Equations}, 62(5):Paper No. 151, 2023.

\bibitem{MR4851904}
K.~Bogdan, M.~Gutowski, and K.~Pietruska-Pa\l{}uba.
\newblock Polarized {H}ardy--{S}tein identity.
\newblock {\em J. Funct. Anal.}, 288(7):Paper No. 110827, 39, 2025.


\bibitem{MR4372148}
K.~Bogdan, T.~Jakubowski, J.~Lenczewska, and K.~Pietruska-Pa\l{}uba.
\newblock Optimal {H}ardy inequality for the fractional {L}aplacian on {$L^p$}.
\newblock {\em J. Funct. Anal.}, 282(8):Paper No. 109395, 31, 2022.

\bibitem{MR3737628}
K.~Bogdan, J.~Rosi\'{n}ski, G.~Serafin, and L.~Wojciechowski.
\newblock L\'{e}vy systems and moment formulas for mixed {P}oisson integrals.
\newblock In {\em Stochastic analysis and related topics}, volume~72 of {\em Progr. Probab.}, pages 139--164. Birkh\"{a}user/Springer, Cham, 2017.

\bibitem{MR4521651}
K.~Bogdan and M.~Wi\c{e}cek.
\newblock Burkholder inequality by {B}regman divergence.
\newblock {\em Bull. Pol. Acad. Sci. Math.}, 70(1):83--92, 2022.

\bibitem{MR4896163}
M.~Bonforte, J.~Dolbeault, B.~Nazaret, and N.~Simonov.
\newblock Stability in {G}agliardo--{N}irenberg--{S}obolev inequalities:
flows, regularity and the entropy method.
\newblock {\em Mem. Amer. Math. Soc.}, 308(1554):viii+166, 2025.

\bibitem{MR4518648}
M.~Borowski and I.~Chlebicka.
\newblock Controlling monotonicity of nonlinear operators.
\newblock {\em Expo. Math.}, 40(4):1159--1180, 2022.

\bibitem{MR1853037}
J.~A. Carrillo, A.~J\"{u}ngel, P.~A. Markowich, G.~Toscani, and A.~Unterreiter.
\newblock Entropy dissipation methods for degenerate parabolic problems and generalized {S}obolev inequalities.
\newblock {\em Monatsh. Math.}, 133(1):1--82, 2001.

\bibitem{MR648601}
K.~L. Chung.
\newblock {\em Lectures from {M}arkov processes to {B}rownian motion}, volume 249 of {\em Grundlehren der Mathematischen Wissenschaften}.
\newblock Springer-Verlag, New York-Berlin, 1982.

\bibitem{MR3443368}
S.~N. Cohen and R.~J. Elliott.
\newblock {\em Stochastic calculus and applications}.
\newblock Probability and its Applications. Springer, Cham, second edition, 2015.

\bibitem{MR745449}
C.~Dellacherie and P.-A. Meyer.
\newblock {\em Probabilities and potential. {B}}, volume~72 of {\em North-Holland Mathematics Studies}.
\newblock North-Holland Publishing Co., Amsterdam, 1982.
\newblock Theory of martingales, Translated from the French by J. P. Wilson.

\bibitem{MR1808433}
W.~Farkas, N.~Jacob, and R.~L. Schilling.
\newblock Feller semigroups, {$L^p$}-sub-{M}arkovian semigroups, and applications to pseudo-differential operators with negative definite symbols.
\newblock {\em Forum Math.}, 13(1):51--90, 2001.

\bibitem{MR2589887}
B.~A. Frigyik, S.~Srivastava, and M.~R. Gupta.
\newblock Functional {B}regman divergence and {B}ayesian estimation of distributions.
\newblock {\em IEEE Trans. Inform. Theory}, 54(11):5130--5139, 2008.

\bibitem{Gneiting2011}
T.~Gneiting.
\newblock Making and evaluating point forecasts.
\newblock {\em J. Amer. Statist. Assoc.}, 106(494):746--762, 2011.

\bibitem{GottwaldBraun2026}
S.~Gottwald and D.~A. Braun.
\newblock On the structure of information.
\newblock Preprint, arXiv:2409.20331v3, 2026.

\bibitem{MR3729529}
T.~Grzywny and M.~Kwa\'{s}nicki.
\newblock Potential kernels, probabilities of hitting a ball, harmonic functions and the boundary {H}arnack inequality for unimodal {L}\'{e}vy processes.
\newblock {\em Stochastic Process. Appl.}, 128(1):1--38, 2018.

\bibitem{MR4622410}
M.~Gutowski.
\newblock Hardy--{S}tein identity for pure-jump {D}irichlet forms.
\newblock {\em Bull. Pol. Acad. Sci. Math.}, 71(1):65--84, 2023.

\bibitem{MR4885983}
M.~Gutowski and M.~Kwa\'snicki.
\newblock Beurling--{D}eny formula for {S}obolev--{B}regman forms.
\newblock {\em Nonlinear Anal.}, 257:Paper No. 113808, 14, 2025.

\bibitem{MR1873235}
N.~Jacob.
\newblock {\em Pseudo differential operators and {M}arkov processes. {V}ol. {I}}.
\newblock Imperial College Press, London, 2001.
\newblock Fourier analysis and semigroups.

\bibitem{MR3558361}
D.~Kim.
\newblock Martingale transforms and the {H}ardy-{L}ittlewood-{S}obolev inequality for semigroups.
\newblock {\em Potential Anal.}, 45(4):795--807, 2016.

\bibitem{MR3010850}
V.~Knopova and R.~L. Schilling.
\newblock A note on the existence of transition probability densities of {L}\'{e}vy processes.
\newblock {\em Forum Math.}, 25(1):125--149, 2013.

\bibitem{MR0126722}
M.~A. Krasnoselski\u{\i} and J.~B. Ruticki\u{\i}.
\newblock {\em Convex functions and {O}rlicz spaces}.
\newblock P. Noordhoff Ltd., Groningen, 1961.

\bibitem{MR4568704}
F.~K\"{u}hn and R.~L. Schilling.
\newblock Maximal inequalities and some applications.
\newblock {\em Probab. Surv.}, 20:382--485, 2023.

\bibitem{MR1224450}
R.~L. Long.
\newblock {\em Martingale spaces and inequalities}.
\newblock Peking University Press, Beijing; Friedr. Vieweg \& Sohn, Braunschweig, 1993.

\bibitem{MR3463679}
C.~Marinelli and M.~R\"{o}ckner.
\newblock On the maximal inequalities of {B}urkholder, {D}avis and {G}undy.
\newblock {\em Expo. Math.}, 34(1):1--26, 2016.

\bibitem{MR2598413}
F.~Nielsen and R.~Nock.
\newblock Sided and symmetrized {B}regman centroids.
\newblock {\em IEEE Trans. Inform. Theory}, 55(6):2882--2904, 2009.

\bibitem{MR2964297}
A.~Os\c{e}kowski.
\newblock {\em Sharp martingale and semimartingale inequalities}, volume~72 of {\em Mathematics Institute of the Polish Academy of Sciences. Mathematical Monographs (New Series)}.
\newblock Birkh\"auser/Springer Basel AG, Basel, 2012.

\bibitem{MR2273672}
P.~E. Protter.
\newblock {\em Stochastic integration and differential equations}, volume~21 of {\em Stochastic Modelling and Applied Probability}.
\newblock Springer-Verlag, Berlin, 2005.
\newblock Second edition. Version 2.1, Corrected third printing.

\bibitem{MR1113700}
M.~M. Rao and Z.~D. Ren.
\newblock {\em Theory of {O}rlicz spaces}, volume 146 of {\em Monographs and Textbooks in Pure and Applied Mathematics}.
\newblock Marcel Dekker, Inc., New York, 1991.

\bibitem{MR1725357}
D.~Revuz and M.~Yor.
\newblock {\em Continuous martingales and {B}rownian motion}, volume 293 of {\em Grundlehren der mathematischen Wissenschaften [Fundamental Principles of Mathematical Sciences]}.
\newblock Springer-Verlag, Berlin, third edition, 1999.

\bibitem{MR274683}
R.~T. Rockafellar.
\newblock {\em Convex analysis}.
\newblock Princeton Mathematical Series, No.~28. Princeton University
Press, Princeton, NJ, 1970.

\bibitem{MR1157815}
W.~Rudin.
\newblock {\em Functional analysis}.
\newblock International Series in Pure and Applied Mathematics. McGraw-Hill, Inc., New York, second edition, 1991.

\bibitem{MR1739520}
K.-i. Sato.
\newblock {\em L\'{e}vy processes and infinitely divisible distributions}, volume~68 of {\em Cambridge Studies in Advanced Mathematics}.
\newblock Cambridge University Press, Cambridge, 1999.
\newblock Translated from the 1990 Japanese original, Revised by the author.

\bibitem{MR252961}
E.~M. Stein.
\newblock {\em Topics in harmonic analysis related to the {L}ittlewood-{P}aley theory}.
\newblock Annals of Mathematics Studies, No. 63. Princeton University Press, Princeton, NJ; University of Tokyo Press, Tokyo, 1970.

\bibitem{MR3206685}
F.-Y. Wang.
\newblock {$\Phi$}-entropy inequality and application for {SDE}s with jumps.
\newblock {\em J. Math. Anal. Appl.}, 418(2):861--873, 2014.

\end{thebibliography}

\section*{Acknowledgments}

We thank Iwona Chlebicka, Agnieszka Ka\l{}amajska,
Mateusz Kwa\'snicki, Adam Os\c{e}kowski,
Artur Rutkowski, and Ren\'e Schilling
for helpful discussions.
We also thank Alex Kulik, Jacek Ma\l{}ecki,
and Konstantin Merz for references.
Finally, we thank the editors of EJP,
whose suggestion prompted us to expand the
discussion of connections with statistical
learning in Section~1.3.

\end{document}